\documentclass{siamltex}

\usepackage[utf8]{inputenc}
\usepackage{amsmath}
\usepackage{graphicx}
\usepackage{amssymb}
\usepackage{listings}
\usepackage{booktabs}
\usepackage{tabularx}
\usepackage{algorithm}
\usepackage{algorithmic}
\usepackage{textcomp}
\usepackage{caption}
\usepackage{subfig}
\usepackage{textcomp}
\usepackage{mathtools}
\usepackage{enumitem}
\usepackage{soul}
\usepackage[dvipsnames]{xcolor}

\usepackage{pgfplots}
\usepackage{tikz}
\usetikzlibrary{pgfplots.groupplots}
\pgfplotsset{compat=1.13}

 	\definecolor{debianred}{rgb}{0.84, 0.04, 0.33}

\DeclareMathOperator*{\argmax}{argmax}
\DeclareMathOperator*{\argmin}{argmin}
\newtheorem{remark}{Remark}
\newcommand{\trace}{{\rm trace}}

\newcommand{\spalm}{sPALM }
\newcommand{\spalmtt}{sPALM-DL-TT }

\newcommand{\spalmdot}{sPALM}
\newcommand{\spalmttdot}{sPALM-DL-TT}

\newcommand{\DL}{DL}

\newcommand{\palmtt}{PALM-DL-TT }
\newcommand{\palmttdot}{PALM-DL-TT}

\newcommand{\bipalmdot}{iPALMbt}

\definecolor{greenf}{rgb}{.054, .5, .005}

\begin{document}

\bibliographystyle{siam}

\title{
A spectral PALM algorithm for matrix and tensor-train based Dictionary Learning
\thanks{Version of \today. The first and the third authors are members of the INdAM Research Group GNCS that partially supported this  work.} }{}
\date{\today}
\author{Domitilla Brandoni\thanks{Dipartimento di Matematica, AM$^2$,
        Alma Mater Studiorum - Universit\`a di Bologna, Piazza di Porta San Donato 5, 
        40126 Bologna, Italia. Emails: 
{\tt  \{domitilla.brandoni2,margherita.porcelli,valeria.simoncini\}@unibo.it}}
\and Margherita Porcelli$^{\dagger,}$\thanks{ISTI--CNR, Via Moruzzi 1, Pisa, Italia}
\and Valeria Simoncini$^{\dagger,}$\thanks{IMATI-CNR, Via Ferrata 5/A, Pavia, Italia}}

\maketitle

\begin{abstract} 
Dictionary Learning (DL) is one of the leading
sparsity promoting techniques in the context of image classification, where the ``dictionary'' matrix
$D$ of images and the sparse matrix $X$ are determined so as to represent a redundant image dataset. 
The resulting constrained optimization problem is nonconvex and non-smooth, providing several
computational challenges for its solution.
To preserve multidimensional data features, various tensor DL formulations 
have been introduced, adding to the problem complexity. 
We develop a new alternating algorithm for the solution of
the DL problem both in the matrix and tensor frameworks; in the latter case
a new formulation based on Tensor-Train decompositions is also proposed.
The new method belongs to the 
Proximal Alternating Linearized Minimization (PALM) algorithmic family, with the inclusion of
second order information to enhance efficiency. We discuss a rigorous convergence analysis, and
report on the new method performance on the image classification of several benchmark datasets.

\end{abstract}

\begin{keywords}
Dictionary Learning, Tensor-Train decomposition, Nonconvex-nonsmooth minimization, Alternating minimization,  Proximal gradient algorithms, Spectral gradient method
\end{keywords}

\begin{AMS}
 65F30, 15A23,  15A69, 65K05, 90C06.
\end{AMS}

\section{Introduction}
Sparse representation of data has become an important tool in a variety of 
contexts such as image classification and compression,
observation denoising and equation solving. In the context of image classification, 
Dictionary Learning (DL) is among the leading
sparsity promoting techniques, and we refer to \cite{Dumitrescu2018} and  \cite{Mairal2014a}
for an overview of all applications of dictionary learning in image processing in general.

Given an array of data $Y$, DL aims to find a matrix $D$ called dictionary and a sparse 
matrix $X$ to represent $Y$ as $Y \approx D X$, under certain constraints on
$D$ and $X$. 
Each column of the dictionary $D$ can be seen as a compressed representation of the 
redundant information contained in $Y$. 
A distinctive feature of this approximate factorization
is that the number of columns of the dictionary $D$, called atoms, is greater than the number 
of rows. A rich number of atoms may have advantages in terms of invariance 
of the dictionary under specific geometric transformations, such as translations or rotations, so 
called ``shiftability'' \cite{shiftability2,Simoncelli1992}.

While originally the problem was formulated so that only $X$ was an unknown data, 
the seminal {works \cite{Aharon2006,olshausen1997}} introduced a different perspective, where both $D$ and $X$ are 
to be determined. From a computational view point, this nonlinearity creates big challenges, especially
when constraints are included. 
To cope with nonlinearity, most DL algorithms rely on {\it alternating} optimization:
the minimization in $X$ is known as \emph{sparse coding} and it is often performed
via the so called Orthogonal Matching Pursuit (OMP), while the minimization in $D$ 
is known as \emph{dictionary update} and various approaches have been proposed. 
The Method of Optimal Directions (\cite{MOD}) computes the dictionary by setting to 
zero the partial gradient in $D$ of the objective function. 
The K-SVD (\cite{Aharon2006}) updates each dictionary atom separately by 
using the Singular Value Decomposition to sequentially obtain a series of best rank-one approximations in
each mode.
To avoid the computation of several 
singular value decompositions, the Approximate K-SVD (AK-SVD) was 
proposed in \cite{AKSVD}. 
An exhaustive overview of the matrix DL algorithms can be found in \cite{Dumitrescu2018}.


Due to the increased need to analyze multidimensional data, various tensor formulations of
the \DL\ problem have been introduced with the aim of preserving data structure and feature
heterogeneity.   For instance, Tensor SVD (tSVD), the Canonical Polydiac Decomposition (CPD) and the High Order SVD
have been used, see, e.g., \cite{Duan2012,Roemer2014,Zhang2015}.
{Each atom of the dictionary is updated separately using the corresponding tensor decomposition.}
In \cite{Zubair2013} the tensor data $\mathcal{Y}$ is modeled using a sparse Tucker Decomposition, the sparsity 
is imposed on the core tensor $\mathcal{X}$, while the dictionary is replaced by the factor matrices in the 
Tucker decomposition. The decomposition is determined by an alternating iteration using a gradient descent method, where
the sparse tensor $\mathcal{X}$ is obtained using a greedy algorithm, named Tensor Orthogonal Matching Pursuit 
(TOMP). 
Another interesting and more recent tensor formulation can be found in \cite{Dantas2019} where the dictionary is represented as a sum of Kronecker products of smaller subdictionaries. 
This is equivalent to imposing a CP structure on the corresponding dictionary tensor. 
{An analogous Kronecker structure was considered in \cite{bajwa1,bajwa2} where the authors 
show that,  given a sufficient number of noisy training columns in $Y$, and 
under certain conditions on the problem parameters, 
the structured dictionary is locally identifiable with high probability.


%
%
%
{When nonconvex and non-smooth models for the DL problem are considered, the tensor-based minimization procedures available in the literature are not usually supported by a theoretical analysis providing global convergence guarantees,
thus limiting this methodology to a purely exploratory setting.}

We advance the algorithmic developments by proposing 
a new {nonconvex and non-smooth} 
theoretically  founded alternating algorithm for the
matrix and tensor formulations of the DL problem. In the tensor case, we
propose a novel application of the Tensor-Train (TT) decomposition leading to a
multi-dimensional dictionary, called TT-DL.
%
%
The new algorithm belongs to the class of proximal alternating linearized minimization (PALM) algorithms and is named spectral PALM (\spalmdot).  The original PALM algorithm was presented in \cite{Bolte2014} and several variants were later proposed \cite{Udell,Gao2020,TITAN,iPALM}
and applied to various problems, including standard matrix based  
\DL\ problem formulations as shown in \cite{Bao2014,palmDL,Zhu2020}.
These methods, called PALM-type algorithms, perform a gradient step in each variable
 and take into account the constraints using proximal maps. 
Remarkably, convergence to critical points is ensured for these algorithms for a large class of 
nonconvex non-smooth problems where the variable vector is split into several blocks of variables. 
PALM-type algorithms are generally based on the use  of Lipschitz constants 
that may be unavailable or hard to estimate, possibly leading to low performance. 
Our new method \spalm differs in the choice of the stepsize as it implicitly embeds second 
order information  of the objective function, so as to take longer steps than 
using the more conservative Lipschitz-based stepsizes. 
As a consequence, \spalm is a PALM-type algorithm with the same 
convergence properties but with an expected better practical performance. 
{The use of enriched proximal steps for solving composite optimization problems, that is problems where the objective is the sum of
a smooth and a non-smooth function, is not new, see, e.g., \cite{Sparsa} and the 
more recent advances using inexact variable metric in \cite{Bonettini2} and Newton-like steps  
in \cite{kanzow2021}. The novelty of our approach consists in constructing spectral stepsizes
that use information from the previous iteration of the alternating algorithm allowing to relate them with local second order information of the smooth part of the objective function.
In particular, when \spalm is applied to \DL, we provide new explicit bounds for the spectral stepsizes that generalize known results for strictly convex quadratic problems \cite{Zanni}.}

In this work we show that PALM-type algorithms (including \spalmdot) can be naturally
 applied to the proposed Tensor-Train \DL\ formulation, yielding 
convergent schemes. Moreover, in the matrix and tensor setting, 
the proposed spectral variant yields better performance in the solution of \DL\ image classification problems.
To the best of our knowledge, these are the first {globally} convergent tensor-based algorithms in the \DL\ literature.

{
This paper is organized as follows. In section~\ref{sec:dl_problem} we describe the matrix DL problem and its use in image classification, and provide new formulations based of the TT decomposition of the dictionary.
Then we illustrate the main algorithmic framework of this work, i.e. the PALM methods, in Section \ref{sec:palm}, and propose sPALM  in Section  \ref{sec:sPALM} where the theoretical analysis is carried-out. 
The application of PALM algorithms to the proposed DL formulation is described in Section~\ref{sec:palm_todl} where their convergence is also proved.  Section~\ref{sec:numerical_expe} is devoted to numerical tests and conclusions are drawn in Section~\ref{sec:conclusions}.  
}



 \subsection*{Notation}
Vectors  and scalars are denoted by lowercase letters $(a,b,\dots)$, matrices are 
denoted by capital letters $(A, B,\dots)$ and higher-order tensors by calligraphic letters $(\mathcal{A}, \mathcal{B},\cdots)$. 
Capital Greek letters (e.g., $\Omega, \Gamma, \Theta$) indicate specific sets of real matrices.
%
In the following, $\|\cdot \|_F$ indicates the matrix Frobenius norm, {$\|A\|_2$} denotes the matrix norm induced by
the Euclidean vector norm, while $\|x\|_0$ denotes the zero-norm of a vector or tensor, defined as the number of its nonzero entries.

\section{The \DL\ problem and a new tensor formulation}
\label{sec:dl_problem}
For a {\it training} set of data, Dictionary Learning consists of solving a two variable optimization problem. 
We are interested in the following formulation:
Given the array 
$Y = [y_1, \dots ,y_p] \in \mathbb{R}^{n\times p} $
and a sparsity threshold $\tau>0$, solve
\begin{equation} \label{eq:dl_matrix}
 \min_{D,X}\|Y-DX\|_F^2 \hspace{5mm} s.t.  \hspace{5mm} D\in\Omega_{n,k}, \hspace{3mm} X\in\Gamma_{k,p}^{(\tau)} ,
\end{equation}
where $k>n$ and
\begin{eqnarray}
\Omega_{n,k} &=& \{ D=[d_1, \ldots, d_k] \in \mathbb{R}^{n\times k}:\ \|d_i\|_2 = 1, i = 1,\dots,k\}
 \label{eq:omega} \\
\Gamma_{k,p}^{(\tau)} &=& 
\{ X=[x_1, \ldots, x_p] \in \mathbb{R}^{k\times p}: \|x_i\|_0 \le \tau,  i = 1,\dots,p \}. \label{eq:gamma} 
\end{eqnarray}
Other formulations, not necessarily equivalent, are possible \cite{Bao2014,Fan21,Mairal2014a}. 
This minimization problem is NP-hard, see e.g. \cite{Dumitrescu2018,Mairal2014a}, and nonconvex. Nonconvexity comes 
from two sources: the sparsity promoting functional $l_0$-norm
and the bi-linearity between the dictionary $D$ and the sparse representation $X$. In addition, the $l_0$-norm
makes the problem non-smooth. 

To preserve the multidimentional structure of the data, a more general {\it tensor} form of the
\DL\ problem can also be used. 
More precisely, consider for instance
a fourth-dimensional array $\mathcal{Y}\in\mathbb{R}^{n_1\times n_2 \times n_e \times n_p}$
consisting of $n_1\times n_2$ images of $n_p$ persons in $n_e$ expressions. 
Using the ${m}\choose {n}$-mode product (see this and related definitions in Appendix \ref{sec:tt_tools})
 and the constraint sets defined in (\ref{eq:omega}) and (\ref{eq:gamma}), the tensor \DL\ problem can be formulated as
\begin{equation}
  \min_{\mathcal{D},\mathcal{X}}\|\mathcal{Y}-\mathcal{D}\times_3^1\mathcal{X}\|_F^2 \hspace{5mm} 
s.t.  \hspace{5mm} \mathcal{D}_{[2]}\in\Omega_{n_1n_2,k}, \hspace{3mm} \mathcal{X}_{[1]}\in\Gamma_{k,n_en_p}^{(\tau)} ,
 \label{eq:dl_tensor}
\end{equation}
where 
 $\mathcal{D}\in\mathbb{R}^{n_1\times n_2\times k}$ is a third-order tensor with unit norm frontal slices,
 and $\mathcal{X}\in\mathbb{R}^{k\times n_e \times n_p}$ is a sparse tensor with at most $\tau$ nonzero elements per column 
fiber, and $k>n_1n_2$. 
All the constraints proper of the matrix setting can be reformulated on the tensors themselves or on their matricizations. 
The formulation (\ref{eq:dl_tensor}) is equivalent to (\ref{eq:dl_matrix}) since
$\|\mathcal{Y}-\mathcal{D}\times_3^1\mathcal{X}\|_F=
\|\left(\mathcal{Y}-\mathcal{D}\times_3^1\mathcal{X}\right)_{[2]}\|_F=
\|\mathcal{Y}_{[2]}-\mathcal{D}_{[2]}\mathcal{X}_{[1]}\|_F$.
To reduce memory requirements, instead of considering the whole tensor $\mathcal{D}$, we  propose
its Tensor-Train (TT) Decomposition (see Definition \ref{def:tt}). 
For third order tensors the TT decomposition can be written using either the 
$m\choose n$-mode product or the $n$-mode product (\ref{def:nmodeprod}) as {
$\mathcal{D}=G_1\times_2^1 \mathcal{G}_2 \times_3^1 G_3
=\mathcal{G}_2 \times_1 G_1 \times_3 G_3^T,$
where $G_1\in\mathbb{R}^{n_1\times r_1}$, $\mathcal{G}_2\in\mathbb{R}^{r_1\times n_2\times r_2}$, $G_3\in\mathbb{R}^{r_2\times k}$ are the TT-cores and $r_1$, $r_2$ are the TT-ranks. }
Following the original TT-SVD algorithm, we require  the columns of $G_1$ and of $(\mathcal{G}_2)_{[2]}$ to be orthonormal. 
This orthogonality property makes
the computation of the Lipschitz constants more convenient (see Proposition \ref{prop:lipschitz_PALMTT}),
and the constraint on $\mathcal{D}$ easier to handle.
To properly define the Tensor-Train formulation of the \DL\ problem the following additional
constraint set of matrices with orthonormal columns is introduced,
 \begin{equation}\label{eq:theta}
 \Theta_{m,n} = \{ G \in \mathbb{R}^{m\times n}: G^TG = I_n\}.
 \end{equation}
Then, the TT formulation of the \DL\ problem takes the form

\begin{eqnarray} \label{eq:dl_tt}
 \min_{G_1, \mathcal{G}_2 ,G_3,\mathcal{X}}
\|\mathcal{Y}-(G_1\times_2^1 \mathcal{G}_2 \times_3^1 G_3)\times_3^1\mathcal{X}\|_F^2 \hspace{3mm}  
s.t.&& \hspace{3mm}  \mathcal{X}_{[1]}\in\Gamma_{k,n_en_p}^{(\tau)} \hspace{2mm} G_1\in\Theta_{n_1, r_1}  \nonumber
 \\
 && \hspace{3mm} (\mathcal{G}_2)_{[2]}\in\Theta_{r_1n_2,r_2} \hspace{2mm} G_3\in\Omega_{r_2,k}.
\end{eqnarray}
By tensor unfolding
$\|\mathcal{Y}-(G_1\times_2^1 \mathcal{G}_2 \times_3^1 G_3)\times_3^1\mathcal{X}\|_F=
\|\mathcal{Y}_{[2]}-\left(I_{n_2}\otimes G_1\right)\left(\mathcal{G}_2\right)_{[2]} G_3\mathcal{X}_{[1]}\|_F$.
The constraints on $G_3$ and $\mathcal{X}_{[1]}$ are inherited from the \DL\ formulation. 
In particular, using the TT formulation, the constraint on the columns of $\mathcal{D}_{[2]}$ in (\ref{eq:dl_tensor}) becomes a constraint on the columns of $G_3$. This can be easily proved using the orthogonality of $G_1$ and $(\mathcal{G}_2)_{[2]}$.

The TT formulation described above can be extended to a multiway tensor $\mathcal{Y}\in\mathbb{R 
}^{n_1\times\dots\times n_q\times\dots \times n_s}$ with $q<s$,
as
\begin{eqnarray} \label{eq:dl_tensor_N}
&&\min_{G_1, \dots ,G_{q+1},\mathcal{X}}\left\|\mathcal{Y}-\left(G_1\times_2^1 \mathcal{G}_2 \times_3^1 \dots \times_3^1 G_{q+1}\right)\times_{q+1}^1\mathcal{X}\right\|_F^2  s.t. \\
&&   \mathcal{X}_{[1]}\in\Gamma_{k,n_{q+1}\dots n_s}^{(\tau)},
 \hspace{2mm} G_1\in\Theta_{n_1, r_1} \hspace{2mm} G_{q+1}\in\Omega_{r_q,k} ,\nonumber
 \\
 && (\mathcal{G}_j)_{[2]}\in\Theta_{r_{j-1}n_j,r_{j+1}} \hspace{2mm} \textup{for}\quad j=2,\dots,q  . \nonumber
\end{eqnarray}
 This formulation can be used when either the dimensionality of the database or the dimensionality of 
the single data is higher than 2; see, e.g., Section~\ref{sec:4D}.

\subsection{The \DL\ classification problem} \label{sec:classification}
Classification is one of the major tasks within data mining;
we refer the reader to \cite[Chapter 8]{Dumitrescu2018} for an overview of different \DL\ classification algorithms. 
Among them, two approaches seem to be highly rated in the \DL\ literature.
In the first, the dictionary $D$ and 
the sparse matrix $X$ are learnt from data by solving (\ref{eq:dl_matrix}) and a 
classifier matrix $W\in\mathbb{R}^{n_p\times k}$  is computed a posteriori  as the solution of the following problem:
\begin{equation} \label{eq:W}
 \min_W \|C-WX\|_F^2+\beta \|W\|_F^2,
\end{equation}
 where $\beta$ is a positive small constant, and the matrix
$C\in\mathbb{R}^{n_p\times n_en_p}$ contains the labels of the images stored in $Y$.
In particular, if the column $i$ of $Y$ contains the person $\ell$, then 
the $i$th column of $C$ is equal to $e_\ell$, the $\ell$th canonical basis vector. 
Different values of $\beta$ were investigated in our computational experiments. Since this
stage of the computation is not our main algorithmic concern, 
we will report numerical results only for $\beta=0$.

In the second category  of classification algorithms 
the classifier $W$ is learnt from the data together with $X$ and $D$:
%
\begin{equation} \label{eq:classification_DWX}
\min_{D,X,W}\|Y-DX\|_F^2+\gamma \|C-WX\|_F^2 \hspace{5mm} s.t.  \hspace{5mm} D\in\Omega_{n,k}, \hspace{3mm} X\in\Gamma_{k,p}^{(\tau)},
\end{equation}
where $\gamma$ is a positive constant. 
This formulation aims to enforce a representative 
strength together with a discriminative action. 
For this reason this classification algorithm has been named ``discriminative \DL''; see, e.g., \cite[sec.8.5.2]{Dumitrescu2018} and
references therein. To maintain the presentation sufficiently concise, we will not further discuss this
second formulation, although all results can be adapted to this setting. We refer to
\cite{Brandoni.PhD.22} for a more detailed presentation.

Once the classifier has been computed, the classification task proceeds as follows:
Given a new image $y\in\mathbb{R}^{n}$ to be classified, 
its sparse representation $x\in\mathbb{R}^{k}$ is computed, e.g., by an OMP-type 
algorithm\footnote{See, e.g., {\tt http://www.cs.technion.ac.il/$\,\tilde{\,}$\,ronrubin/software.html}.
In the tensor case, we derived a new tensor-train version of OMP  
{named OMP-TT where the sparse solution is computed using the TT-cores of the dictionary without 
explicitly creating $\mathcal{D}$. For further details see  \cite[Chapter 7]{Brandoni.PhD.22}}.},
so that $y$ is assigned to class $\hat{\ell}={\rm argmax}_i |Wx|_i$.
 
Preserving the multidimensional structure of the data can be extremely useful also 
in classification contexts. 
As in the previous section, the optimization problem  (\ref{eq:classification_DWX}) involving the classifier matrix can be generalized to the (multi-order) tensor setting.

\section{The general PALM framework}\label{sec:palm} 
We first review the main properties of the original 
PALM algorithm as a general platform for \DL\ oriented PALM-type algorithms. Then we
introduce the new spectral PALM method (hereafter \spalmdot) and its convergence properties.
To this end, we also need to recall certain general aspects of non-smooth nonconvex optimization and
to fix our assumptions.

The Proximal Alternating Linearized Minimization (PALM) algorithm \cite{Bolte2014} provides
a general setting for solving non-smooth nonconvex optimization problems of the form
 \begin{equation} \label{eq:nnpb}
  \min_{x,y} \Psi(x,y) \hspace{5mm} \textup{with} \hspace{5mm}  \Psi(x,y):= H(x,y)+f_1(x)+f_2(y) ,
 \end{equation}
 where the functions $f_1$ and $f_2$ are extended valued (i.e., allowing the inclusion of constraints) and $H$ is a smooth 
 coupling function, only required to have partial Lipschitz continuous gradients
 $\nabla_x H$ and $\nabla_y H$  (see more precise definitions later on).
This optimization problem has become a reference tool in many machine learning and image processing methodologies, see, e.g., the examples mentioned in \cite{iPALM}.
%
For each block of coordinates in  (\ref{eq:nnpb}), PALM performs one gradient step 
on the smooth part, followed by a proximal step on the non-smooth part.
The method belongs to the class of Gauss-Seidel proximal schemes, also known as
alternating minimization schemes, and generalizes to the nonconvex non-smooth case well-known and widely used
alternating algorithms \cite{Bau2011,beck2017first,Bonettini,GrippoSciandro,iPALM}.

An important contribution to the success of PALM was the convergence proof strategy obtained
{in \cite{Attouch2010,Bolte2014}}. This allowed the design of new convergent alternating minimization algorithms, consisting of
a sequence
converging to critical points of (\ref{eq:nnpb}).


The PALM algorithm relies on the knowledge of the partial Lipschitz moduli of $\nabla_x H$ and $\nabla_y H$
or of some upper estimates, and was 
applied to sparse non-negative matrix factorizations in \cite{Bolte2014},
 for which partial Lipschitz moduli are explicitly available, though its practical behavior was not investigated. 
Inertial variants of PALM have been later proposed with the aim of accelerating the convergence of 
the original algorithm \cite{Gao2020,TITAN,iPALM}.
All these variants enjoy the convergence properties of the original PALM, and are based on the 
Lipschitz constants explicitly available for all the addressed applications. 
When these constants are not explicitly known, then a backtracking scheme can be employed 
to approximate their action \cite{beck2017first,FISTA}, so that convergence results still hold.
%
%
%
Our {\it nonconvex-nonsmooth} setting provides significant challenges. We
consider problems of the form (\ref{eq:nnpb}), for which we assume that
 the functions $H$, $f_1$ and $f_2$ satisfy the following minimal assumptions set.
 
\vskip 0.05in
{\sc Assumption A}
 \begin{description}
  \item(A1) $f_1: \mathbb{R}^n \rightarrow (-\infty,+\infty]$ and $f_2: \mathbb{R}^m  \rightarrow (-\infty,+\infty]$ are proper and lower semicontinuous functions such that $\inf_{\mathbb{R}^{n}}{f_1}>-\infty$ and $\inf_{\mathbb{R}^{m}}{f_2}>-\infty$.
\item(A2) $H:  \mathbb{R}^{n\times m} \rightarrow  \mathbb{R}$ is continuously differentiable and $\inf_{\mathbb{R}^{n\times m}} \Psi>-\infty$.
 \item{(A3) $\nabla H$ is Lipschitz continuous on bounded subsets of $\mathbb{R}^{n\times m}$.}
 \item(A4) {The partial gradients $\nabla_x H(x,y)$ and $\nabla_y H(x,y)$  are  globally Lipschitz continuous}, i.e. there exist nonnegative $L'_1(y)$ and $L'_2(x)$ such that
 \begin{equation*}
 \mbox{ fixed } y, \ \  \left\|\nabla_x H(u,y)-\nabla_x H(v,y)\right\|_2\leq L'_1(y)\left\|u-v\right\|_2, \ \forall u,v \in \mathbb{R}^{n}  ,
 \end{equation*}
 \begin{equation*}
 \mbox{ fixed } x, \ \   
\left\|\nabla_y H(x,u)-\nabla_y H(x,v)\right\|_2\leq L'_2(x)\left\|u-v\right\|_2, \ \forall u,v \in \mathbb{R}^{m} .
 \end{equation*}
 \item(A5) There exist  $\lambda_{i}^{+}>0$, $i=1,2$ such that 
\begin{equation}
  \label{eq:limit_lipschitz_sup}
   \sup \{L_1(y^k); \, k\in\mathbb{N}\}\leq \lambda_1^{+} , \hspace{8mm}  \sup \{L_2(x^k); \, k\in\mathbb{N}\}\leq \lambda_2^{+}.
 \end{equation}

 \end{description}
\vskip 0.1in
{We call  {\em partial smoothness parameters} the constants $L'_1(y)$ and $L'_2(x)$ in 
Assumption~A4.
From the definition of Lipschitz continuity  it follows that if a function is Lipschitz continuous with smoothness 
parameter $L$ then it is also  Lipschitz continuous with any $\bar L \ge L$. }
As usual, we will call {\em Lipschitz constant} the smallest possible smoothness parameter of a given function;
for the partial gradients of the function $H$ it will be denoted as  $L_1(y)$ and $L_2(x)$.

\begin{remark}
{
In \cite{Bolte2014} a further assumption is made to ensure that 
the constants $L'_1(y)$ and $L'_2(x)$ 
are uniformly bounded away from zero, i.e. that  there 
exist $\lambda_{i}^{-}>0$, $i=1,2$ such that
\begin{equation}
  \label{eq:limit_lipschitz_inf}
\inf \{L_1'(y^k); \, k\in\mathbb{N}\}\geq \lambda_1^{-} , \hspace{8mm}  
\inf \{L_2'(x^k); \, k\in\mathbb{N}\}\geq \lambda_2^{-}.
 \end{equation}
In our description we avoid this assumption by choosing 
the partial smoothness parameters $L'_1$ and $L'_2$ safely bounded away from 
zero, as suggested in \cite[Remark~3]{Bolte2014}. 
}
\end{remark}


 
The PALM algorithm with constant stepsize is reported in Algorithm \ref{alg:PALM}.
In there, the standard Moreau proximal mapping  is employed (\cite{beck2017first}):
Given a lower and semicontinuos function $\sigma \, : \mathbb{R}^{m}\rightarrow (-\infty, \infty]$ and a scalar $t>0$, the proximal map is defined as follows
 \begin{equation} \label{eq:prox_map}
  \textup{prox}_t^{\sigma} (x) =\argmin_{w\in\mathbb{R}^{m}}\left\{ \sigma (w)+\frac{t}{2} \|w-x\|^2\right\}.
 \end{equation}
%

\begin{algorithm}[ht]
\caption{PALM (with constant stepsize) \label{alg:PALM}}
\begin{algorithmic}[1]
\STATE {\bf{Input}}: 
$(x_0,y_0)\in \mathbb{R}^{n} \times \mathbb{R}^{m}$, $\eta_1,\eta_2>1$,  $\mu_1, \mu_2 >0$
\FOR {$k = 0,1,\dots, $}
\STATE {\emph{Update $x$}:
Set  {$L''_1(y_k) = \max \{\eta_1 L'_1(y_k), \mu_1\}$ and $\bar \alpha_{k,1} = 1/L''_1(y_k)$} and compute
 \begin{equation}\label{stepx}
 x_{k+1} = \textup{prox}^{f_1}_{1/\bar \alpha_{k,1}}\left (x_k - \bar \alpha_{k,1} \nabla_{x} H(x_k,y_k) \right)
\end{equation}}
\STATE {\emph{Update $y$}:
Set  {$L''_2(x_{k+1}) = \max \{\eta_2 L'_2(x_{k+1}), \mu_2\}$ and $\bar \alpha_{k,2} = 1/L''_2(x_{k+1})$} and compute
\begin{equation}\label{stepy}
y_{k+1} = \textup{prox}^{f_2}_{1/\bar \alpha_{k,2}}\left (y_k - \bar \alpha_{k,2} \nabla_{y} H(x_{k+1},y_k) \right)
\end{equation}}
\ENDFOR
\end{algorithmic}
\end{algorithm}

PALM alternates the minimization on the two blocks $(x,y)$ and makes explicit use of the 
smoothness parameters $L''_1(y_k)$ and  $L''_2(x_k)$.
When these parameters are not available, they can be approximated by using a backtracking strategy.
Indeed, setting $\Psi_1(x,y) := H(x,y)+f_1(x)$ and $L_{0,1}=1$, at each iteration $k>1$,  
 the procedure starts with  $L_{k,1}=L_{k-1,1}$ and then $L_{k,1}$ is increased  by a 
constant factor, typically doubled, until the following  {\it sufficient decrease condition} is met
\begin{equation}\label{eq:Psi}
\Psi_1(x_{k+1}, y_k)  \le \Psi_1(x_k,y_k) + \langle \nabla_x H(x_k,y_k),x_{k+1}-x_k  \rangle + \frac{L_{k,1}}{2} \| x_{k+1}-x_k\|_2^2.
\end{equation}
Then, the stepsize $\bar \alpha_{k,1}$ is taken as the reciprocal of the found  
value (analogously for $\bar \alpha_{k,2}$, where we set $\Psi_2(x,y) := H(x,y)+f_2(y)$).
%
%
%
The following theorem reports the main convergence result proved for the original PALM algorithm and
successively extended to all PALM-type algorithms
\footnote{The result can be generalized to the case
of a function $\Psi$ that satisfies the so-called Kurdyka-\L{}ojasiewicz property, as done for PALM in \cite{Bolte2014}.
}, see \cite{Bolte2014,TITAN,Gao2020,iPALM}.

\begin{theorem}
\label{thm:convergence_PALM}
Suppose that $\Psi$ is semi-algebraic such that Assumptions~A hold. Let $\{z_k\} =\{(x_k,y_k)\}$ be a bounded sequence generated by PALM in Algorithm~\ref{alg:PALM}. Then the sequence $\{z_k\}$ has finite length, that is
$\sum_{k=1}^\infty \left \| z_{k+1}-z_k\right \| < \infty,$ and converges to a critical point $z^*$ of $\Psi$.
\end{theorem}

\section{The spectral PALM algorithm (\spalmdot)}\label{sec:sPALM}
Starting from the PALM framework, we propose \spalm that, for each coordinate block,  
employs  a {\em spectral} gradient step in the smooth part of the operator, while maintaining a proximal step for the
non-smooth part.
More precisely, \spalm uses a spectral stepsize in each variable bock in combination with an Armijo-type 
backtracking strategy ensuring the overall convergence.
Spectral\footnote{The denomination ``spectral'' refers to the property that the steplength is related to the spectrum of the 
average Hessian matrix (when well-defined).}
gradient methods are well-known optimization strategies for the solution of large scale  unconstrained 
and constrained optimization problems \cite{surveyPBB,PBB}. These
algorithms are rather appealing for their simplicity, low-cost per iteration (gradient-type algorithms) and good practical performance due to a clever choice of the step length.
The key to the success of these approaches, also known as Barzilai-Borwein methods from the pioneering work \cite{BB88}, 
lies in the explicit use of first-order information of the cost function
on the one hand and, on the other hand, in the implicit use of second-order information 
embedded in the step length through a rough approximation of the cost function Hessian.
While spectral gradient methods were first proposed for convex quadratic problems, they have
have been widely used in a large variety of more general contexts \cite{Zanni,RBB,treni}.

Spectral stepsizes have also been used in \cite{Bonettini} 
 in the context of {\it convex} constrained optimization problems.
Indeed the proposed Cyclic Block Coordinate Gradient Projection algorithm in \cite{Bonettini} makes use of spectral stepsizes  when applied to non-negative matrix factorizations.
  However, these steps are used to determine an 
approximate solution of the minimization problem for each variable block,  and 
they do not use information from the previous iteration of the alternating algorithm. 
To the best of our knowledge, the use of spectral stepsizes embedded in an alternating algorithm for non-convex non-smooth problems of the form (\ref{eq:nnpb}) has remained so far unexplored.  We contribute to fill this gap.

{
Classical spectral stepsizes, also known as BB stepsizes {from the initials of Barzilai 
and Borwein}, are motivated by the quasi-Newton approach, 
where the  inverse of the Hessian matrix is replaced by a multiple of the identity 
matrix \cite{BB88}. Consider the case of the partial  Hessian 
$\nabla_{xx} H$ (the case of  $\nabla_{yy} H$ is analogous).
For a given iteration $k$, let $s_k=x_{k+1}-x_k $ and $g_k=\nabla_x H(x_{k+1},y_{k})-\nabla_x H(x_{k},y_k) $ be the difference between two consecutive iterates and corresponding gradient values.
Then 
$\nabla_{xx} H(x_{k+1},y_{k})$ is approximated by ${\alpha_{k+1}^{-1}}I$ where the positive scalar $\alpha_{k+1}$ is defined by 
either of the following BB values 
$$
\alpha_{k+1}^{BB1} = \argmin_{\alpha} \| {\alpha^{-1}} s_k - g_k\| \quad \mbox{ or } \quad \alpha_{k+1}^{BB2} = \argmin_{\alpha} \| s_k - {\alpha}g_k\|
$$
that is,
\begin{equation}\label{eq:bb12}
\alpha_{k+1}^{BB1} = \frac{ \langle s_{k},s_{k} \rangle }{\langle s_{k},g_{k}\rangle} \quad \mbox{ or } \quad  \alpha_{k+1}^{BB2} = \frac{\langle s_{k},g_{k}\rangle}{\langle g_{k},g_{k}\rangle}.
\end{equation}
}
A variety of different rules based on suitable adaptive combinations of  
$\alpha_{k+1}^{BB1}$ and $\alpha_{k+1}^{BB2}$
in (\ref{eq:bb12}) have been proposed in the literature  {in the solution of ``single block variable"  nonlinear optimization problems   and  it was observed experimentally that alternating the two stepsizes along iterations is  beneficial for the performance, 
see \cite{Zanni,treni} and references therein}. 
We report in Algorithm \ref{alg:BBstep}  a simple alternating rule based on \cite{GS}
that gives the best results in our numerical experiments (see 
Section~\ref{sec:numerical_expe}). Other rules can be equally {adapted} within \spalmdot. 
The inclusion of threshold values in Algorithm~\ref{alg:BBstep} ensures that the $\alpha$'s remain bounded. 
The overall \spalm scheme is reported in Algorithm~\ref{alg:sPALM}.

\begin{algorithm}[htb]
\caption{Computation of the spectral stepsize $\alpha_{k+1,i}$, $i = 1$ or $i=2$ \label{alg:BBstep}}
\begin{algorithmic}[1]
\STATE {\bf{Input}}: $s_{k,i}, g_{k,i}$,  $0 < \alpha_{\min} \le \alpha_{\max}$.

\IF { $ \langle s_{k,i},g_{k,i} \rangle  > 0 $}
\IF{$k$ is odd}
\STATE {$\alpha_{k+1,i}= \max \left\{\alpha_{\min}, \min \left\{\alpha_{k+1,i}^{BB1},\alpha_{\max} \right\} \right\}$ with $\alpha_{k+1,i}^{BB1}=  \frac{\langle s_{k,i},s_{k,i}\rangle }{\langle s_{k,i},g_{k,i}\rangle } 
$}\label{eq:bb1}
\ELSE
\STATE {$
\alpha_{k+1,i} = \max \left\{\alpha_{\min}, \min \left\{\alpha_{k+1,i}^{BB2},\alpha_{\max} \right\} \right\}$ with $\alpha_{k+1,i}^{BB2}=\frac{ \langle s_{k,i},g_{k,i} \rangle }{ \langle g_{k,i},g_{k,i} \rangle }$}  
 \label{eq:bb2}
\ENDIF
\ELSE
 \STATE$  \alpha_{k+1,i} = 1$

\ENDIF
\end{algorithmic}
\end{algorithm}

\begin{algorithm}[ht]
\caption{\spalmdot \label{alg:sPALM}}
\begin{algorithmic}[1]
\STATE {\bf{Input}}: 
$(x_0,y_0)\in \mathbb{R}^{n} \times \mathbb{R}^{m}$, $\rho_1,\delta_1,\rho_2,\delta_2  \in (0,1)$,
$0 < \alpha_{\min} \le \alpha_{\max}$,  $\alpha_{0,1}, \alpha_{0,2}\in [\alpha_{\min}, \alpha_{\max}]$.
\FOR {$k = 0,1,\dots, $}
\STATE {\emph{Update $x$}:
Set \begin{equation}\label{stepx_spalm}
 x_{k+1} = \textup{prox}^{f_1}_{1/\bar \alpha_{k,1}}\left (x_k - \bar \alpha_{k,1} \nabla_{x} H(x_k,y_k) \right)
\end{equation}}
where $\bar \alpha_{k,1} = \rho_1^{i_k} \alpha_{k,1}$ and $i_{k}$ is the smallest nonnegative integer for which the following condition is satisfied,
\begin{equation}\label{lsx}
\Psi_1(x_{k+1}, y_k)  \le \Psi_1(x_k,y_k) - \frac{\delta_1}{2 \bar \alpha_{k,1}} \| x_{k+1}-x_k\|_2^2
\end{equation}
\STATE Compute $\alpha_{k+1,1} \in [\alpha_{\min}, \alpha_{\max}]$ using Algorithm \ref{alg:BBstep} with  $s_{k,1} = x_{k+1}-x_k$ and $g_{k,1} =  \nabla_{x} H(x_{k+1},y_k) -  \nabla_{x} H(x_k,y_k)$. \label{alpha1}
\STATE {\emph{Update $y$}:
Set \begin{equation}\label{stepy_spalm}
y_{k+1} = \textup{prox}^{f_2}_{1/\bar \alpha_{k,2}}\left (y_k - \bar \alpha_{k,2} \nabla_{y} H(x_{k+1},y_k) \right)
\end{equation}
where $\bar \alpha_{k,2} = \rho_2^{j_k} \alpha_{k,2}$  and $j_{k}$ is the smallest nonnegative integer for which the following condition 
is satisfied,
\begin{equation}\label{lsy}
\Psi_2(x_{k+1}, y_{k+1})  \le \Psi_2(x_{k+1},y_k) - \frac{\delta_2}{ 2 \bar \alpha_{k,2}} \|y_{k+1}-y_k\|_2^2
\end{equation}}
\STATE Compute $\alpha_{k+1,2} \in [\alpha_{\min}, \alpha_{\max}]$ using Algorithm \ref{alg:BBstep} with  $s_{k,2} = y_{k+1}-y_k$ and $g_{k,2} =  \nabla_{y} H(x_{k+1},y_{k+1}) -  \nabla_{y} H(x_{k+1},y_k)$. \label{alpha2}
\ENDFOR
\end{algorithmic}
\end{algorithm}

\begin{remark}\label{rem:finite}
{
Under Assumption A, conditions (\ref{lsx}) and (\ref{lsy}) in Algorithm~\ref{alg:sPALM} are satisfied in a finite number of backtracking  steps.
For instance,  from {the {\em descent lemma}, see e.g. \cite[Lemma 1]{Bolte2014},} 
 for any $x$ defined by $  x = \textup{prox}^{f_1}_{1/ \alpha}\left (x_k - \alpha \nabla_{x} H(x_k,y_k) \right)$
we have that 
$$
\Psi_1(x, y_k)  \le \Psi_1(x_k,y_k) - \frac{1}{2}\left(\frac{1}{\alpha} -L_1(y_k)\right ) \| x-x_k\|_2^2,
$$
and condition (\ref{lsx}) is satisfied for $\alpha \le  \frac{1-\delta_1}{L_1(y_k)}$.
 Therefore, backtracking terminates with $\bar \alpha_{k,1}\! \ge\! \min\! \left\{\alpha_{k,1},\frac{\rho_1(1-\delta_1)}{L_1(y_k)}\right \} 
\!\ge \! \min \left\{\alpha_{\min},\frac{\rho_1(1-\delta_1)}{\lambda_1^+}\right \}$. Similarly for condition (\ref{lsy}).
}
\end{remark}
\vskip 0.1in



In the following we set up the theoretical tools for proving a convergence result analogous 
to that of Theorem~\ref{thm:convergence_PALM} {by exploiting the proof of methodology
introduced in \cite{Bolte2014} (see also \cite[Section 3]{iPALM})}.
Standard notation and definitions of non-smooth analysis will be used, see, e.g., \cite{nonsmooth}.
The convergence of \spalm is then a consequence of the convergence analysis carried 
out in \cite{Bolte2014} and relies on the following lemma. 
 
 \begin{lemma}\label{lem:C2}
Suppose   {Assumption~A  holds}. Let $\{z_k\} =\{(x_k,y_k)\}$ be a bounded sequence generated 
by \spalm from a starting point $z_0$ and let $\omega(z_0)$  be the set of all limit points of $\{z_k\}$. 
Then the following conditions hold.
\begin{description}
\item[C1)] There exists a positive scalar $\gamma_1$ such that
$\gamma_1 \|z_{k+1}-z_k\|_2^2 \le \Psi(z_k) - \Psi(z_{k+1})$;
 \item[C2)] There exists a positive scalar $\gamma_2$ such that
for some $w_k \in \partial \Psi(z_k)$ we have 
$\|w_k\|_2 \le \gamma_2 \|z_k - z_{k-1}\|, \mbox{ for } k = 0,1, \dots$;
\item[C3)] Each limit point in the set $\omega(z_0)$ is a critical point for $\Psi$. 
\end{description}
\end{lemma}
 \begin{proof}
{We first observe that the stepsizes $\bar \alpha_{k,1}$ and
 $\bar \alpha_{k,2}$ remain bounded for all $k$. Indeed,  since $\alpha_{k,1}\in [\alpha_{\min}, \alpha_{\max}]$ we have that  $\bar \alpha_{k,1} = \rho_1^{i_k} \alpha_{k,1} \le \alpha_{\max}$ as $\rho_1 \in (0,1)$ and $i_k \ge0$.
Moreover, $\bar \alpha_{k,1}$ and $\bar \alpha_{k,2}$ are 
uniformly bounded from below, see Remark \ref{rem:finite}.

Item C1) can be proved as follows. Fix $k \ge0$ and sum  the 
inequalities (\ref{lsx}) and (\ref{lsy}), so as to obtain
 \begin{eqnarray*}
\Psi_1(x_{k+1}, y_k) +  \Psi_2(x_{k+1}, y_{k+1})  &\le &\Psi_1(x_k,y_k) + \Psi_2(x_{k+1},y_k) \\
& & \quad - \frac{\delta_1}{2 \bar \alpha_{k,1}} \| x_{k+1}-x_k\|_2^2
- \frac{\delta_2}{ 2 \bar\alpha_{k,2}} \|y_{k+1}-y_k\|_2^2.
\end{eqnarray*}
Recalling the definition of $\Psi_1$ and $\Psi_2$ around (\ref{eq:Psi}) we obtain
 \begin{eqnarray*}
\Psi(x_{k+1}, y_{k+1}) & \le & \Psi(x_k,y_k) - \frac{\delta_1}{2 \bar \alpha_{k,1}} \| x_{k+1}-x_k\|_2^2 - \frac{\delta_2}{ 2 \bar\alpha_{k,2}} \|y_{k+1}-y_k\|_2^2\\
 & \le & \Psi(x_k,y_k) - \frac{\delta_1}{2 \alpha_{\max}} \| x_{k+1}-x_k\|_2^2 - \frac{\delta_2}{ 2 \alpha_{\max}} \|y_{k+1}-y_k\|_2^2 ,
\end{eqnarray*}
from which  Item C1) follows, with $\gamma_1 = \frac{1}{2 \alpha_{\max}}\min \left \{ \delta_1,\delta_2 \right \}$.

The proofs of items C2) and C3) follow the lines of the proofs of \cite[Lemmas 4 and 5]{Bolte2014}.
Indeed,  from the definition of the proximal 
map and the iterative steps (\ref{stepx_spalm}) and (\ref{stepy_spalm}), we have that
$$
 x_{k} = \argmin_{x\in \mathbb{R}^n} \left\{\langle \nabla_x H(x_{k-1}, y_{k-1}),  x-x_{k-1}\rangle +  \frac{1}{2 \bar \alpha_{k,1}} \|x- x_{k-1}\|_2^2 + f_1(x)\right\}
$$
and
$$
 y_{k} = \argmin_{y\in \mathbb{R}^m} \left\{ \langle \nabla_y H(x_{k}, y_{k-1}),y-y_{k-1}\rangle +  \frac{1}{2 \bar \alpha_{k,2}} \|y- y_{k-1}\|_2^2 + f_2(y)\right\} .
$$
Using the fact that $\bar \alpha_{k,1}$ and  $\bar \alpha_{k,2}$ are bounded for all $k$ the results follow.}
 \end{proof}

We can now state a convergence result for \spalmdot.
\begin{theorem}
\label{thm:convergence_PALM_LS}
Suppose that $\Psi$ is semi-algebraic such that   {Assumption~A  holds}. 
Let $\{z_k\} =\{(x_k,y_k)\}$ be a bounded sequence generated by \spalmdot. Then the sequence 
$\{z_k\}$ has finite length
and converges to a critical point $z^*$ of $\Psi$.
\end{theorem}
\begin{proof}
The proof follows by using Lemma~\ref{lem:C2} and applying the proof methodology proposed in \cite{Bolte2014}.
\end{proof}

\vskip 0.1in
\begin{remark}
Algorithm \ref{alg:PALM} and Algorithms \ref{alg:BBstep}-\ref{alg:sPALM} can be extended to 
the general setting involving $p>2$ blocks, that is problems of the form
  \begin{equation}\label{multiterm}
  \min_{x_i\in \mathbb{R}^{n_i}} \Psi(x_1, \dots, x_p) := H(x_1,\dots, x_p)+\sum_{i=1}^p f_i(x_i),
 \end{equation}
 for which Theorems \ref{thm:convergence_PALM} and  \ref{thm:convergence_PALM_LS} hold. 
 When variable blocks are matrices, all PALM-type algorithms can be extended 
to the matrix optimization setting by 
using the trace matrix scalar product and the Frobenius norm in place of the 
vector scalar product and the vector 2-norm, respectively.
\end{remark}
\vskip 0.1in

\section{Application of the PALM framework to \DL} \label{sec:palm_todl}
%
The PALM methodology can be applied to the \DL\ problem  in the matrix setting; see, e.g.,
\cite{Bao2014,Zhu2020,palmDL}. We provide a general framework for \DL\ leading to
convergent schemes, that can also be employed in the case of the tensor formulation.

In the matrix case, problem (\ref{eq:dl_matrix}) can be equivalently formulated as
\begin{equation}\label{eq:dl_ind}
 \min_{D,X}\|Y-DX\|_F^2+\delta_{\Omega_{n,k}}(D)+\delta_{\Gamma_{k,p}^{(\tau)}}(X),
\end{equation}
where  $\delta_{\Omega_{n,k}}$and $\delta_{\Gamma_{k,p}^{(\tau)}}$ are  indicator functions over the sets $\Omega_{n,k}$ and $\Gamma_{k,p}^{(\tau)}$ defined in (\ref{eq:omega}) and (\ref{eq:gamma}), respectively. 
Given a non-empty and closed set $\Omega\subseteq \mathbb{R}^{m\times n}$, we recall that
the indicator function $\delta_{\Omega}:\mathbb{R}^{m \times n}\rightarrow (-\infty,+\infty]$ is given by
 \begin{equation}
 \delta_{\Omega}(A)=
    \begin{cases}
      0 & \text{if $A\in\Omega$}\\
      +\infty & \text{otherwise} .
    \end{cases}       
\end{equation}

The formulation (\ref{eq:nnpb}) has clearly the form (\ref{eq:dl_ind})  with $H(D,X)=\|Y-DX\|_F^2$, $f_1(D)=\delta_{\Omega_{n,k}}(D)$ and $f_2(X)=\delta_{\Gamma_{k,p}^{(\tau)}}(X)$. 
For these functions, it has been proved  in \cite{Bao2014,palmDL}
that Assumptions~A1-A4 hold and that the Lipschitz moduli for the partial
gradients $\nabla_X H$ and $\nabla_D H$ are given, respectively, by
\begin{equation}\label{eq:lipmat}
L_X = 2\|D^TD\|_2 \quad \mbox{ and } \quad  L_D =2 \|X X^T\|_2.
\end{equation}

The overall objective function is semi-algebraic: $H$ is a real polynomial function, while 
$f_1$ and $f_2$ 
are indicator functions of semi-algebraic sets, and thus semi-algebraic as well.
Assuming that the sequence of iterates generated by PALM is bounded and being $H$  twice continuously differentiable, PALM 
is thus guaranteed to converge to 
a critical point of problem (\ref{eq:dl_ind}),  see Theorem~\ref{thm:convergence_PALM}. The same argument 
can be applied to
\spalm invoking Theorem~\ref{thm:convergence_PALM_LS}, and to the other PALM variants.
We next give the explicit form of the BB stepsizes,
where $(X_k,D_k)$ play the role of $(x_k,y_k)$ in Algorithms \ref{alg:BBstep}-\ref{alg:sPALM}, and  show  some key bounds.

{
\vskip 0.1in
\begin{proposition}\label{rem:alfaL}
Let $S_{k} = X_{k+1} - X_{k}$ and $T_k = D_{k+1} - D_{k}$ be computed at the $k$th iteration of  
sPALM applied to problem (\ref{eq:dl_ind}). Assume that 
$\langle S_k, (D_k^T D_k) S_k \rangle \neq0 $ and $\langle T_k,  T_k (X_{k+1} X_{k+1}^T)\rangle\neq0$.
Then the BB stepsizes take the form
\begin{eqnarray*}
 \alpha_{X_{k+1}}^{BB1}  = & \frac{1}{2}\frac{ \langle S_k,S_k \rangle }{\langle S_k, (D_k^T D_k) S_k \rangle}  \ \mbox{ and }  \
\alpha_{X_{k+1}}^{BB2}  = &   \frac{1}{2} \frac{\langle S_k, (D_k^T D_k) S_k\rangle}{\langle  (D_k^T D_k) S_k, (D_k^T D_k) S_k \rangle}, \\
 \alpha_{D_{k+1}}^{BB1}  = & \frac{1}{2} \frac{ \langle T_k,T_k \rangle }{\langle T_k, T_k (X_{k+1} X_{k+1}^T) \rangle} \ \mbox{ and } \
\alpha_{D_{k+1}}^{BB2}  = &  \frac{1}{2} \frac{\langle T_k, T_k (X_{k+1} X_{k+1}^T) \rangle}{\langle  T_k (X_{k+1} X_{k+1}^T) ,T_k (X_{k+1} X_{k+1}^T)  \rangle },
\end{eqnarray*}
and the following bounds hold
$$
{\frac 1 {L_{X_k}} \le \alpha_{X_{k+1}}^{BB2} \le  \min \left \{ \alpha_{X_{k+1}}^{BB1}, \frac {1}{2\sigma_{\min}^2(D_k)}\right \} } 
$$
and
\begin{equation}
{
\frac 1 {L_{D_{k+1}}} \le \alpha_{D_{k+1}}^{BB2} \le \min \left \{ \alpha_{D_{k+1}}^{BB1}, \frac {1}{2\sigma_{\min}^2({X_{k+1}})} \right \} \label{ineq_D}
}
\end{equation}
where $\sigma_{\min}(D_k)$ and  $\sigma_{\min}({X_{k+1}})$ are the smallest nonzero singular values of $D_k$ and ${X_{k+1}}$, respectively, and $L_{D_k}$ and $L_{X_{k+1}}$ are given in (\ref{eq:lipmat}).
\end{proposition}
\begin{proof}
From the definitions of $\alpha_{k+1}^{BB1}$ and $\alpha_{k+1}^{BB2}$ in (\ref{eq:bb12}) 
and from observing that $\nabla_X H = - 2D^T(Y-DX)$ and $\nabla_D H =- 2(Y-DX)X^T$, we get the form
of the BB stepsizes $ \alpha_{X_{k+1}}^{BB1},  \alpha_{X_{k+1}}^{BB2},  \alpha_{D_{k+1}}^{BB1}$ and $\alpha_{D_{k+1}}^{BB2}$. Moreover, the partial Hessians of $H$ have the form $\nabla_{XX}H = I \otimes2(D^TD)$, $\nabla_{DD}H =2(XX^T) \otimes I$ and are positive semidefinite. Therefore the 2-norm  Lipschitz constants are $L_X = \lambda_{\max} (\nabla_{XX}H)$ and $L_D = \lambda_{\max}(\nabla_{DD}H).
$

Let us consider the stepsizes $\alpha_{X_{k+1}}^{BB1}$ and $\alpha_{X_{k+1}}^{BB2}$.
We observe that they are the reciprocal of Rayleigh quotients 
for $2(D_k^T D_k)$ and $2(D_k D_k^T)$, using $S_k$ and $D_k S_k$, respectively, as 
$
\alpha_{X_{k+1}}^{BB2} = \frac{1}{2} \|D_k S_k\|_F^2/\langle  D_k S_k, (D_k D_k^T) (D_k S_k) \rangle
$.
Moreover, 
\begin{eqnarray*}
\alpha_{X_{k+1}}^{BB2} & = &\frac{1}{2} \frac{\langle S_k,  S_k\rangle}{\langle S_k, (D_k^T D_k) S_k \rangle}\frac{\langle S_k, (D_k^T D_k) S_k\rangle^2}{\langle  (D_k^TD_k) S_k, (D_k^T D_k) S_k \rangle\langle S_k,  S_k\rangle}\\
& = & \alpha_{X_{k+1}}^{BB1}  \frac{\| S_k\|_F^2 \| (D_k^T D_k) S_k\|_F^2 \cos^2\phi_k}{\| S_k\|_F^2 \|( D_k^T D_k) S_k\|_F^2} = \alpha_{X_{k+1}}^{BB1}\cos^2\phi_k
\end{eqnarray*}
where $\phi_k$ is the angle between {the vectorization $\hat s_k$ of $S_k$ and  $(I\otimes (D_k^TD_k))\hat s_k$}.  
Therefore,
$$
\alpha_{X_{k+1}}^{BB1}\ge  \alpha_{X_{k+1}}^{BB2} \ge \frac{1}{\lambda_{\max}(2 (D_k^TD_k))} = \frac{1}{L_{X_k}}.
$$
Finally, let ${\tt D}_k = D_k^TD_k\ge 0$. Then 
{\small $ \alpha_{X_{k+1}}^{BB2} = \frac{1}{2}  \|{\tt D}_k^\frac 1 2 S_k\|_F^2/\|({\tt D}_k^\frac 1 2 ){\tt D}_k^\frac 1 2 S_k\|_F^2 \le 1/(2\sigma^{2}_{\min}({D_k}))$},
where the last inequality follows from the fact that ${\tt D}_k^\frac 1 2 S_k$ belongs to the range of ${\tt D}_k$.

The inequalities for $ \alpha_{D_{k+1}}^{BB1} $ and $ \alpha_{D_{k+1}}^{BB2} $ in (\ref{ineq_D})  can be derived analogously.
\end{proof}
}
\vskip 0.1in

Proposition~\ref{rem:alfaL} {gives bounds on 
the BB stepsizes for the case of semidefinite Hessian and shows that  longer steps are made {than with the} 
reciprocal of the Lipschitz constants.
These bounds generalize to the singular case known bounds for strictly quadratic functions, see,
 e.g., \cite{Zanni}.
In the strictly convex quadratic case the BB stepsizes are able to sweep the spectrum 
of the Hessian matrix yielding faster convergence than using a standard steepest-descent method.
  An analogous faster convergence is expected in practice in the (only) convex case.}



Similarly, this approach can be applied  to the classification problem (\ref{eq:classification_DWX}),
see \cite{Brandoni.PhD.22}.
 
We next illustrate the main contribution of this section by showing the
 application of the PALM framework to the tensor formulation (\ref{eq:dl_tt}), which
can be rewritten as

{\footnotesize
\begin{equation}
\min_{G_1, \mathcal{G}_2 ,G_3,\mathcal{X}} H(G_1, \mathcal{G}_2 ,G_3,\mathcal{X}) +\delta_{\Gamma_{k,n_en_p}^{(\tau)}}(\mathcal{X}_{[1]}) + \delta_{\Theta_{n_1,r_1}}(G_1) + \delta_{\Theta_{r_1 n_2,r_2}}((\mathcal{G}_2)_{[2]})+ \delta_{\Omega_{r_2, k}}(G_3),
 \label{eq:PALM_tt}
\end{equation}
}
where 
$H(G_1, \mathcal{G}_2 ,G_3,\mathcal{X})=
\|\mathcal{Y}-(G_1\times_2^1 \mathcal{G}_2 \times_3^1 G_3)\times_3^1\mathcal{X}\|_F^2$. 
%
%
%
We determine explicit values\footnote{If $G_1, G_2$ did not have orthonormal
columns, upper bounds for the Lipschitz constants could still be obtained.}
for the Lipschitz constants of the partial gradient of 
$H$ in  (\ref{eq:PALM_tt}), where 
$L_S$ is the Lipschitz constant of $H$ corresponding to the variable $S$ in the $\|\cdot\|_2$ norm.
 \begin{proposition}
 \label{prop:lipschitz_PALMTT}
 Set $G_2=\left(\mathcal{G}_2\right)_{[2]}$, $X=\mathcal{X}_{[1]}$, $Y=\mathcal{Y}_{[2]}$ and consider 
  \begin{equation*}
  H(G_1, G_2 ,G_3,X)=\left\|Y-\left(I_{n_2}\otimes G_1\right)G_2 G_3X\right\|_F^2.
  \end{equation*}
 Then the partial gradients of $H$ are globally Lipschitz. For $G_1, G_2$ having orthonormal columns,
 the Lipschitz constants satisfy
$$
L_X = 2\|G_3\|_2^2, \quad L_{G_2} = 2 \|G_3X\|_2^2, \quad L_{G_3} =  2 \|X\|_2^2 ,
$$
and $L_{G_1} =  2\|\sum_{i=1}^p \left(A_i A_i^T\right)\|_2$, where $A_i\in\mathbb{R}^{r_1\times n_2}$ is the 
matricization of the $i$th column of  $A=G_2 G_3 X$.
%
 \end{proposition}

 \begin{proof}
  By direct computation the following expressions for the partial gradients of $H$ hold:
\begin{eqnarray*}
\nabla_X H &= & {2\, \left(\left(I_{n_2}\otimes G_1\right)G_2 G_3\right)^T(-Y + \left(I_{n_2}\otimes G_1\right)G_2 G_3 X)} \\
\nabla_{G_1} H &=& 2 \sum_{i=1}^p (Y_i+G_1 A_i)A_i^T \\
  \nabla_{G_2}H &=&{2 \left(I_{n_2}\otimes G_1\right)^T \left(-Y + \left(I_{n_2}\otimes G_1\right)  G_2 G_3 X\right)  (G_3 X)^T} \\
 \nabla_{G_3}H &=&{2 G_2^T \left(I_{n_2}\otimes G_1\right)^T \left (-Y  + \left(I_{n_2}\otimes G_1\right)G_2 G_3 X\right )  X^T}
\end{eqnarray*}
where 
$A_i\in\mathbb{R}^{r_1\times n_2}$ {and  $Y_i\in\mathbb{R}^{n_1\times n_2} $ are the matricization of the $i$th column of  $A=G_2 G_3 X$ and $Y$, respectively}.
 
Using these expressions we show that each partial gradient is globally Lipschitz.  
In doing so, we derive Lipschitz constants 
by exploiting the orthonormality of the columns of $G_1$ and $G_2$.
For any $\hat{X}$ and $\tilde{X}$ we get 
  \begin{eqnarray}
  \label{eq:nabla_x}
 &&\|\nabla_X H(G_1, G_2 ,G_3,\hat{X})- \nabla_X H(G_1, G_2 ,G_3,\tilde{X}) \|_2 \nonumber \\
  &&=2\|G_3^TG_2^T(I_{n_2}\otimes G_1)^T(I_{n_2}\otimes G_1)G_2 G_3 (\hat{X}-\tilde{X})\|_2 \\
  &&=2\|G_3^T G_3 (\hat{X}-\tilde{X})\|_2 
\le L_X \|\hat{X}-\tilde{X}\|_2, \nonumber 
  \end{eqnarray}
with $L_X = 2 \left\|G_3^T G_3\right\|_2=2\|G_3\|_2^2$. 
In a similar manner we obtain the following inequalities for $\nabla_{G_i}H$, $i=1,2,3$. 

In particular, for any $\hat{G_1}$ and $\tilde{G_1}$ we get 
 \begin{eqnarray*}
  &&\|  \nabla_{G_1} H(\hat{G}_1, G_2 ,G_3,X)- \nabla_{G_1} H(\tilde{G}_1, G_2 ,G_3,X)\|_2\\
&&= 2 \|\sum_{i=1}^p (\hat{G}_1-\tilde{G}_1)A_i A_i^T \| \le L_{G_1} \|(\hat{G}_1-\tilde{G}_1)\| , 
  \end{eqnarray*}
with $L_{G_1}= 2\|\sum_{i=1}^p A_i A_i^T\|_2$. Moreover, for any   $\hat{G_2}$ and $\tilde{G_2}$,  we have
  \begin{eqnarray}
  \label{eq:nabla_g2}
   &&\|  \nabla_{G_2} H(G_1, \hat{G}_2 ,G_3,X)-
  \nabla_{G_1} H(G_1, \tilde{G}_2 ,G_3,X) \|_2\nonumber  \\
  &&= 2 \|(I_{n_2}\otimes G_1)^T (I_{n_2}\otimes G_1)  (\hat{G}_2-\tilde{G}_2) G_3 X  (G_3 X)^T\|_2\\
  &&= 2 \|(\hat{G}_2-\tilde{G}_2) G_3 X  (G_3 X)^T\|_2 \le 
    L_{G_2} \|(\hat{G}_2-\tilde{G}_2)\|_2,  \nonumber
  \end{eqnarray}
where $L_{G_2} = 2\|G_3 X X^T G_3^T\|_2=2\|G_3X\|_2^2$.
Finally, for any $\hat{G_3}$ and $\tilde{G_3}$ we have
\begin{eqnarray} \label{eq:nabla_g3}
   &&\|  \nabla_{G_3} H(G_1, G_2 ,\hat{G}_3,X)- \nabla_{G_3} H(G_1, G_2 ,\tilde{G}_3,X) \|_2 \nonumber \\
  && =2 \left\| G_2^T\left(I_{n_2}\otimes G_1\right)^T \left(I_{n_2}\otimes G_1\right)G_2 
(\hat{G}_3-\tilde{G}_3) X X^T\right\|_2 \nonumber \\
  && =2 \| (\hat{G}_3-\tilde{G}_3) X X^T\|_2 
\le
   L_{G_3} \|\hat{G}_3-\tilde{G}_3\|_2 , \nonumber 
\end{eqnarray}
where 
$L_{G_3} = 2 \|X X^T\|_2=2 \|X\|_2^2$.
 \end{proof}


As an implementation remark, we observe that the computation of the partial gradients 
can avoid the explicit calculation of the Kronecker product $I\otimes G_1$. Indeed,
 using the orthogonality of $I\otimes G_1$ and the properties of the Kronecker product, the 
partial gradient $\nabla_X H$ in Proposition~\ref{prop:lipschitz_PALMTT} above can be computed as
 $\nabla_X H = G_3^TG_2^T (-\tilde{G}^T  + G_2 G_3 X)$,
where $\tilde{G}=\mathtt{reshape}(G_1' \mathtt{reshape}(Y,[n_1,n_2 p]),[r_1 n_2,p])$.
The same applies to all the partial gradients in Proposition~\ref{prop:lipschitz_PALMTT}.

\vskip 0.1in

{Since PALM-type algorithms at each iteration require the computation of a proximal step in each variable block}, 
we now give the formal expression of proximal operators for the three indicator functions in (\ref{eq:PALM_tt}). 
We recall that the proximal map of an indicator function $\delta_{\Omega}$ over a non-empty and 
closed set $\Omega \subseteq \mathbb{R}^{m\times n} $ is the multi-valued projection $P_{\Omega}: \mathbb{R}^{m\times n}\rightrightarrows \Omega$ 
such that 
 \begin{equation}
 \label{eq:projection}
 P_{\Omega}(A)=\argmin_{B\in\Omega}{\|A-B\|_F}.
\end{equation}
Notice that if $\Omega$ is a convex set the projection map is single-valued. The following proposition gives
a closed form expression for the corresponding projection operators over the sets $\Theta_{m,n}$,  $\Omega_{m,n}$ and  $\Gamma_{m,n}^{(\tau)}$.

%

\begin{proposition}
 Let $A\in\mathbb{R}^{m\times n}$.
 \begin{enumerate}[label=\roman*.]
  \item Let  ${\Theta_{m,n}}\subseteq \mathbb{R}^{m\times n}$ be defined  in  (\ref{eq:theta}).
Then $P_{\Theta_{m,n}}(A)=UV^T$,
where $U$ and $V$ are respectively the left and right singular matrices of $A$. 
\item Let $\Omega_{m,n}\subseteq \mathbb{R}^{m\times n}$ be defined  in  (\ref{eq:omega}).
Then $P_{\Omega_{m,n}}(A)=AS^{-1}$,
where $S\in\mathbb{R}^{n\times n}$ is a diagonal matrix whose diagonal elements are the norm of the columns of $A$.
\item  Let $\Gamma_{m,n}\subseteq \mathbb{R}^{m\times n}$  be defined  in  (\ref{eq:gamma}). Then 
$P_{\Gamma_{m,n}}(A)={\tt hard}_{\tau}(A)$,
where  ${\tt hard}_{\tau}$ is the 
hard-thresholding function that selects the $\tau$ largest elements (in absolute value) of each column of $A$ and zeroes all the others. 
 \end{enumerate}
\end{proposition}
{
\begin{proof}
For $B \in \Omega\subset {\mathbb R}^{m\times n}$ having unit norm columns, we have that
$\|A-B\|_F^2=\trace(A^T A)-2\trace(B^T A)+n$ and thus solving (\ref{eq:projection}) is equivalent to finding
 $B=\argmax_{\Omega} \trace(A^T B)$.

$i)$ For $\Omega=\Theta_{m,n}$, let $A=U\Sigma V^T$ be the SVD of $A$, and let $\bar{n}=\min \{m,n\}$. Then 
 \begin{equation*}
  \trace(B^T A)=\trace(B^T U\Sigma V^T)= 
\trace(V^T B^T U \Sigma)=\sum_{i=1}^{\bar{n}} z_{ii} \sigma_i \leq \sum_{i=1}^{\bar{n}} \sigma_i,
 \end{equation*}
where  $z_{ii}$ is the $i$th diagonal element of $Z=V^T B^T U$ and $\sigma_i$ are the singular values of $A$. 
In particular, note that $z_{ii}\le 1$ as $e_i^T V^T B^T U e_i\le \|Ve_i\| \|B\| \|Ue_i\|=1$.
The upper bound is reached for $Z$ equal to the identity matrix, which is obtained for $B=UV^T$.

$ii)$ For $\Omega=\Omega_{m,n}$ and $A=[a_1, \ldots, a_n]$, $B=[b_1, \ldots, b_n]$ with $\|b_j\|=1$, $j=1, \ldots, n$, it
holds that
\begin{equation*}
 \trace(B^TA)=\sum_{j=1}^{n} b_j^T a_j =\sum_{j=1}^{n} \cos(\theta_j) \|a_j\|_2 \leq \sum_{j=1}^n \|a_j\|_2,
\end{equation*}
where $\theta_j\in \left[0,\frac{\pi}{2}\right]$, 
and the upper bound is reached for $b_j=a_j/\|a_j\|_2$. 

$iii)$ The statement follows from \cite[Section 4]{Bolte2014}.
\end{proof} }

Algorithm~\ref{alg:PALM_tt} describes the general PALM scheme with constant stepsize 
applied to the tensor formulation (\ref{eq:PALM_tt}), that is therefore named \palmttdot.
The Lipschitz moduli $L_X, L_{G_i}$ $i = 1,2,3$ are as introduced in Proposition~\ref{prop:lipschitz_PALMTT}. 
This algorithm will be a reference competitor for our new method in the
reported numerical experiments.
The TT algorithmic version of \spalm can be easily obtained following Algorithms \ref{alg:sPALM} and \ref{alg:BBstep} and replacing the stepsizes $\bar \alpha_1,\bar  \alpha_2,\bar  \alpha_3$ and $\bar \alpha_X$ with the spectral stepsizes based on 
Algorithm~\ref{alg:BBstep} and imposing the Armijo sufficient decrease condition.

\begin{algorithm}[t]
\caption{\palmtt}
\label{alg:PALM_tt}
\begin{algorithmic}[1]
\STATE {\bf{Input}}: 
Data matrix $\mathcal{Y}\in\mathbb{R}^{n_1\times n_2\times n_3\times n_e \times n_p}$, 
initial values for $G_1\in\mathbb{R}^{n_1\times r_1}$, 
$\mathcal{G}_2\in\mathbb{R}^{r_1 \times n_2 \times r_2}$, 
$G_3\in\mathbb{R}^{r_2\times k }$ and $\mathcal{X} \in\mathbb{R}^{k\times n_e\times n_p}$, 
maximum number $\tau$ of non-zero elements of 
each fiber of $\mathcal{X}$,  $\eta_1, \eta_2, \eta_3, \eta_X>1$, {$\mu_1, \mu_2, \mu_3, \mu_X >0$}.
\STATE {Set $G_2\leftarrow\left(\mathcal{G}_2\right)_{[2]}$ and $X\leftarrow\left(\mathcal{X}\right)_{[1]}$}
\STATE Compute $G_1=P_{\Theta_{n_1,r_1}}(G_1)$, $G_2=P_{\Theta_{r_1n_2,r_2}} (G_2)$,  
$G_3=P_{\Omega_{r_2,k}}(G_3)$, $X=P_{\Gamma_{k,n_e n_p}^{(\tau)}} (X)$.
\REPEAT 
\STATE {\emph{Update $G_1$}: Set {$L'_{G_1} = \max \{\eta_1 L_{G_1}, \mu_1\}$}  and 
$\bar \alpha_1=1 / L'_{G_1}$ and compute 
$G_1=P_{\Theta_{n_1,r_1}} \left(G_1-\bar \alpha_1 \nabla_{G_1} H\right)$;
}

\STATE {\emph{Update $G_2$}: Set  {$L'_{G_2} =  \max \{\eta_2 L_{G_2}, \mu_2\}$}  and $\bar \alpha_2=1 /L'_{G_2}$ and compute 
$G_2=P_{\Theta_{r_1n_2,r_2}} \left(G_2-\bar \alpha_2 \nabla_{G_2} H\right)$;
}

\STATE {\emph{Update $G_3$}: Set {$L'_{G_3} =  \max \{\eta_3 L_{G_3}, \mu_3\}$}  and $\bar \alpha_3=1 /L'_{G_3}$ and compute 
$G_3=P_{\Omega_{r_2,k}} \left(G_3-\bar \alpha_3\nabla_{G_3} H\right)$;
}

\STATE {\emph{Update $X$}: Set {$L'_{X} =  \max \{\eta_X L_{X},{ \mu_X}\}$}  and $\bar \alpha_X=1 / L'_{X}$ and compute 
$X=P_{\Gamma_{k,n_e n_p}^{(\tau)}} \left(X-\bar \alpha_X\nabla_{X} H\right)$;
}
\UNTIL convergence

\end{algorithmic}
\end{algorithm}

{%
\begin{remark}
The same considerations of Proposition \ref{rem:alfaL}
carry over to the BB stepsizes and the Lipschitz constants for the TT-DL problem  (\ref{eq:PALM_tt}). 
\end{remark}
}

The next theorem contains our main convergence result  for \palmtt described  in Algorithm \ref{alg:PALM_tt}. 
This result can be generalized to 
{\it any} PALM-type algorithm applied to the TT \DL\ formulation  (\ref{eq:PALM_tt}).

\begin{theorem} \label{thm:convergence_general}
 If the sequence generated by \palmtt in Algorithm \ref{alg:PALM_tt} is bounded, then it converges to a critical point of
the TT formulation of the \DL\ problem (\ref{eq:PALM_tt}) and it has finite length property.
\end{theorem}
\begin{proof}
The result is proved by observing that {the function $H$ in 
(\ref{eq:PALM_tt}) is twice continuously differentiable}, 
the functions in the objective of problem (\ref{eq:PALM_tt}) are semi-algebraic and satisfy Assumption~A.
Indeed, all indicator functions are proper, lower semi-continuous and semi-algebraic: the sets $\Omega_{n,m}$ and  $\Gamma_{k,n_en_p}$ are semi-algebraic (see \cite{Bao2014,palmDL}) and 
$\Theta_{n,m}$ is a closed semi-algebraic set for any $n$ and $m$. Also, 
Proposition~\ref{prop:lipschitz_PALMTT} ensures that Assumptions~A3-A5 hold (see \cite[Remark 3]{Bolte2014}).
 \end{proof}


\section{Numerical experiments} \label{sec:numerical_expe}
{We numerically explore the advantages of using the spectral variant of PALM  and 
the possible benefits  of using a TT based algorithm in the solution of the image classification problem.
{We report experiments using the first classification strategy described in Section \ref{sec:classification}, that is the DL problem is first solved and then the classification matrix $W$ is determined by solving (\ref{eq:W}) 
The considered algorithms are compared in terms of efficiency and classification performance 
both in the matrix and tensor settings. A truncated approach for the TT formulation is also
tested. Section~\ref{sec:4D} is devoted to the treatment of 5th order tensors, leading
to a 4D implementation of our algorithms; numerical experiments on a suitable database are reported, illustrating
the effectiveness of the TT approach.}

\subsection{Description of the databases}
We consider four different databases composed of grayscale images with $n=n_1n_2$ pixels
of $n_p$ persons or objects in $n_e$ expressions, where by expression in most cases
we mean different illuminations, view 
angles, etc. 
Each database can thus be naturally represented as a $4$th order tensor 
$\mathcal{Y}\in\mathbb{R}^{n_1\times n_2 \times n_e \times n_p}$. 
The characteristics of all databases are summarized in Table \ref{tab:car_data}.
\begin{enumerate}
 \item MIT-CBCL\footnote{Copyright 2003 -2005 Massachusetts Institute of Technology. All Rights Reserved.} 
{\cite{mit}} is composed by 3240 grayscale images of 10 persons in 324 different expressions. Each image is  
reduced\footnote{To apply the dictionary learning formulation the number of pixels needs to be smaller than the total number of images, that is, 
 $n<n_en_p$. For this reason a shrunk version of the 
MIT-CBCL and Extended Yale databases are considered, where the total number of pixels 
is drastically reduced.} to $15\times15$ pixels. 
\item Extended Yale {\cite{extended}}  consists of more than 16,000 images of 28 subjects in 585 expressions. 
Each image is reduced to $20\times 15$ pixels.
\item MNIST {\cite{mnist}} contains $28\times 28$ size images of ten handwritten digits (from 0 to 9) split in a training set 
composed by 60,000 images and a test set composed by 10,000 images. 
The number of ``expressions'' for each digit varies. In Table~\ref{tab:car_data} the minimum and maximum 
number of expressions is reported. For our experiments we use 2676 expressions for the training 
set and 892 for the test set. 
Notice that the variable expression is not well-defined, i.e., there is no correspondence with two different 
digits in the $j$th expression.
 \item Fashion-MNIST {\cite{fashion}}  contains 70,000 images of 10 different kinds of Zalando’s articles;
3,000 images per item for the training set and 1,000 for the test set were used. 
As for MNIST, the variable ``expression'' is not well-defined.
\end{enumerate}
Each database of $n_p$ persons in $n_e$ expressions is split into 75\% training and 25\% test sets, 
so that
$\tilde{n}_e=0.75n_e$ is the total number of expressions used for training. 

 \begin{table}[hbt]
{\footnotesize
\centering
\begin{tabular}{ccccc }
\toprule
 Database &  pixel size (original) & pixel size (resized) & $n_e$ & $n_p$  \\
\midrule
 MIT-CBCL &$200\times 200$ & $15\times 15$ & $324$ & $10$ \\ 
Ext'd Yale shrunk &$640\times 480$ & $20\times 15$  & $585$ & $28$ \\ 
MNIST &$28\times 28$ &$28\times 28$ & $[892, 6742]$ & $10$ \\ 
Fashion MNIST &$28\times 28$ &$28\times 28$  & $7000$ & $10$ \\
\bottomrule
\end{tabular}
\caption{Pixel size, number of expressions and persons of all databases.}
\label{tab:car_data}
}
\end{table} 

\subsection{Experimental setting}
We consider different PALM-type algorithms for
the solution of the matrix and tensor \DL\ problems in (\ref{eq:dl_matrix}) and (\ref{eq:dl_tt}), respectively. 
All the numerical experiments were conducted on one
node HPE ProLiant DL560 Gen10 with 4 Intel(R) Xeon(R) Gold 6140 CPU @ 2.30GHz and 100 of the 512 Gb of RAM using Matlab  R2019a. We consider the {\em classification rate}  as performance  measure, that is, the percentage of correctly classified 
persons or objects over the total number of test images.

\begin{table}
{\footnotesize
\begin{center}
{
\begin{tabular}{ll}
\toprule 
       & \multicolumn{1}{c }{matrix \DL\ problem (\ref{eq:dl_matrix})} \\
        \midrule
 PALM-DL& PALM  using computation of the exact Lipschitz constant  \cite{Bolte2014}  \\
 iPALM-DL& iPALM using computation of the exact Lipschitz constant \cite{iPALM}\\
 iPALMbt-DL& iPALM  using backtracking to estimate the Lipschitz constant \cite{iPALM}\\
\spalmdot-DL& \spalm based on Algorithms ~\ref{alg:BBstep}-\ref{alg:sPALM}  \\
\bottomrule \\

\toprule 
       & \multicolumn{1}{c }{TT-\DL\ problem (\ref{eq:dl_tt})} \\
        
\midrule 
\palmtt& PALM using computation of the exact Lipschitz constant as in Algorithm \ref{alg:PALM_tt}\\
\spalmtt& \spalm based on Algorithms ~\ref{alg:BBstep}-\ref{alg:sPALM} \\
\bottomrule
\end{tabular}
}
\caption{All methods employed in our experiments.\label{codes}}
\end{center}
}
\end{table}
As starting approximations for the tested algorithms, a random dictionary $\mathcal{D}\in\mathbb{R}^{n_1\times n_2\times k}$ with unit norm frontal slices and a  sparse random tensor $\mathcal{X}\in\mathbb{R}^{k\times n_e\times n_p}$ with at most $\tau$ non-zero elements  per column fiber were used. 
Then, for the matrix formulation we set $Y=\mathcal{Y}_{[2]}$,
 $D=\mathcal{D}_{[2]}$ and $X=\mathcal{X}_{[1]}$. A similar choice is made also for the Tensor-Train setting where $G_1$, $\mathcal{G}_2$ and $G_3$ are initialized as the TT-cores of $\mathcal{D}$.
The parameter $k$ is set such that  {$n<k<\tilde{n}_e n_p$} as is common in dictionary learning.  More precisely, we set $k=441$ for the face databases, $k=1225$ for the Fashion  
MNIST and $k=1600$ for the MNIST.  The sparsity parameter $\tau$ depends on the number of classes, $n_p$, of the database and is set to $4n_p$ for all the databases except MNIST for which $\tau=2n_p$. Other choices of these parameters have been explored. However, we just report the results for these values of $k$ and $\tau$ which seem to exhibit higher classification rate. 

{
We consider the PALM-type algorithms described in Table~\ref{codes}.
The iPALM algorithm is a variant of PALM, where inertial steps are
computed to accelerate convergence \cite{iPALM}.  More precisely, given some positive scalars $\xi_k^1$ and  $\xi_k^2$, 
 the iterate updates in (\ref{stepx}) and (\ref{stepy}) in Algorithm \ref{alg:PALM} are modified as
\begin{eqnarray*}
 \tilde x_k & = & x_k + \xi^1_k (x_k-x_{k-1}), \\
  x_{k+1} & = & \textup{prox}^{f_1}_{1/\bar \alpha_{k,1}}\left (\tilde x_k - \bar \alpha_{k,1} \nabla_{x} H(\tilde x_k,y_k) \right)  ,
\end{eqnarray*}
and analogously for the update of the $y$ block  using $\xi_k^2$. 
The choice of $\xi_k^1$ and  $\xi_k^2$ is crucial both for the convergence and the
acceleration speed of iPALM; setting $\xi_k^1=\xi_k^2=0$ corresponds to the original PALM. Following  the analysis in \cite{iPALM}  for the nonconvex case, also the value of $\eta_1$ and $\eta_2$ should suitably
depend on $\xi_k^1$ and  $\xi_k^2$, but we found that in practice the choice  $\eta_1=\eta_2=1$ yields
much better performance. These values have also been adopted for PALM-DL and PALM-DL-TT based on Algorithm \ref{alg:PALM} while the $\mu$'s are set equal to $10^{-10}$.
Different values for the $\xi_k$'s  have been tested based on the experience in \cite{iPALM} and we report here results for the best performing ones, that is $\xi_k^1=\xi_k^2=0.2$. 
The iPALMbt-DL implementation\footnote{The implementation of iPALMbt-DL  
is an adaptation to the DL problem of the code provided by the courtesy of the authors of \cite{iPALM} for the solution of sparse nonnegative factorizations.} is a variant of iPALM where the Lipschitz constants are approximated using the backtracking strategy (doubling the attempted values) and decreasing the approximated value if a step yielded sufficient decrease
(see also \cite{iPALM-algo}).}

Regarding  implementations based on the spectral variant in  Algorithms~\ref{alg:BBstep}-\ref{alg:sPALM} 
we set: $\alpha_{\min}= 10^{-10}$, $\alpha_{\max}= 10^{10}$,
the initial stepsizes $\alpha_0$'s are set equal to 1, $\rho$'s are set equal to $0.5$ and $\delta$'s equal to $10^{-4}$.


%


{
Deriving a reliable criterion for terminating the iteration is a crucial step towards
the development of a robust method. Most DL implementations in the literature rely on the number of iterations as stopping criterion. 
We have also implemented this choice by setting the maximum number 
of iterations equal to 50 in the forthcoming experiments, where the focus is on discussing the ability of the algorithms in classifying images. Nonetheless,  we have further investigated the use of a problem-driven
 stopping criterion in the analysis of the convergence history of the PALM-type algorithms and report the obtained results in  Appendix \ref{app:test}.}

\subsection{Preliminary tests on the matrix \DL\ {classification problem}}
In this section we want to explore the potential of the 
spectral gradient step compared to that based on the Lipschitz constants in the solution of  the matrix \DL\ classification problem. 
 To this end we compare \spalmdot-DL with the original PALM-DL 
and the inertial variants iPALM-DL and  \bipalmdot-DL on the four datasets described in Table~\ref{tab:car_data}.


{
For the considered methods and all datasets, we report in Table \ref{tab:classbipalm2}  the classification performance after 50 iterations and we plot in Figure \ref{fig:res_bipalm_maxit} the value of $H(D,X)=\|Y-DX\|_F$ as the 
CPU time proceed.

Focusing on the existing PALM variants, we observe that PALM-DL, iPALM-DL
and \bipalmdot-DL reached similar classification performance 
but \bipalmdot-DL is more time consuming. The only exception
is in the classification of the MNIST data set for which \bipalmdot-DL gains roughly the 10\% of classification rate but still at the cost of a higher CPU time (see Figure \ref{fig:res_bipalm_maxit}).
On the other hand sPALM obtains similar percentages for  MIT-CBCL
and Extended Yale  but higher classification performance for the largest data sets, i.e. MNIST and Fashion MNIST, and being in fact the fastest PALM variant.
}

 \begin{table}[hbt]
\centering{
\begin{tabular}{lcccc}
\toprule
        &  \multicolumn{1}{c }{PALM-DL}  &  \multicolumn{1}{c }{iPALM-DL}  & \multicolumn{1}{c}{\bipalmdot-DL}  & \multicolumn{1}{c}{\spalmdot-DL}   \\ \midrule
 MIT-CBCL & $100\%$ &   $100\%$  &$100\%$  & $100\%$ \\ 
Ext'd Yale shrunk & $93.7\%$ & $93.8\%$   & $94.2\%$  & $93.4\%$ \\
MNIST & $64.6\%$ & $65.3\%$  & $76.6\%$ & $80.4\%$ \\ 
Fashion MNIST &$70.0\%$ &  $70.7\%$ &  $71.6\%$ & $74.3\%$ \\ 
\bottomrule
\end{tabular}}
\caption{Successful classification rates of  PALM-DL, iPALM-DL, \bipalmdot-DL, \spalmdot-DL on four different databases.}
\label{tab:classbipalm2}
\end{table}
%
%
\begin{figure}[htb]
\centering
\begin{tikzpicture}
		\begin{axis}[width=.5* \textwidth, 
		   title = {\small MIT-CBCL}, 
		   legend columns=4,
           legend entries={ \small PALM-DL,  \small iPALM-DL, \small \bipalmdot-DL, \small \spalmdot-DL },
           legend to name=named9,
		   xlabel = {\small  CPU time (secs)}, 
		   ylabel = {\scriptsize $\|Y-DX\|_F$},
		   height = .2 \textheight,
		   xmin = -.1,
		 ]
		  \addplot+[mark=none,OrangeRed,line width=1.5,loosely dotted,mark options={ fill=OrangeRed}] table[x index = 1, y index = 2] {data_rev/MIT_DX_time_res_maxit50_tutti};
        \addplot+ [mark=none,PineGreen,line width=1.5,dashed,mark options={ fill=PineGreen}]table[x index = 3, y index = 4] {data_rev/MIT_DX_time_res_maxit50_tutti};
        \addplot+ [mark=none,Blue,line width=1.5,dashdotted,mark options={ fill=Blue}]table[x index = 5, y index = 6] {data_rev/MIT_DX_time_res_maxit50_tutti};
		   \addplot+ [mark=none,Peach,line width=1.8,mark options={ fill=Peach}]table[x index = 7, y index = 8] {data_rev/MIT_DX_time_res_maxit50_tutti};	    
		\end{axis}
\end{tikzpicture}	
\begin{tikzpicture}
		\begin{axis}[width=.5* \textwidth, 
		   title = {\small Extended Yale}, 
		    xlabel = {\small  CPU time (secs)},  
y tick label style={/pgf/number format/sci},
		   height = .2 \textheight,
		   		   xmin = -.1,
		 ]
		 \addplot+[mark=none,OrangeRed,line width=1.5,loosely dotted,mark options={ fill=OrangeRed}] table[x index = 1, y index = 2] {data_rev/YALE_DX_time_res_maxit50_tutti};
        \addplot+ [mark=none,PineGreen,line width=1.5,dashed,mark options={ fill=PineGreen}]table[x index = 3, y index = 4] {data_rev/YALE_DX_time_res_maxit50_tutti};
        \addplot+ [mark=none,Blue,line width=1.5,dashdotted,mark options={ fill=Blue}]table[x index = 5, y index = 6] {data_rev/YALE_DX_time_res_maxit50_tutti};
		   \addplot+ [mark=none,Peach,line width=1.8,mark options={ fill=Peach}]table[x index = 7, y index = 8] {data_rev/YALE_DX_time_res_maxit50_tutti};
		\end{axis}
\end{tikzpicture}
\begin{tikzpicture}
		\begin{axis}[width=.5* \textwidth, 
		   title = {\small MNIST}, 
		   xlabel = {\small  CPU time (secs)},  
y tick label style={/pgf/number format/sci},
		   height = .2 \textheight,
		   		   xmin = -.1,
		 ]
		 \addplot+[mark=none,OrangeRed,line width=1.5,loosely dotted,mark options={ fill=OrangeRed}] table[x index = 1, y index = 2] {data_rev/MNIST_DX_time_res_maxit50_tutti};
        \addplot+ [mark=none,PineGreen,line width=1.5,dashed,mark options={ fill=PineGreen}]table[x index = 3, y index = 4] {data_rev/MNIST_DX_time_res_maxit50_tutti};
        \addplot+ [mark=none,Blue,line width=1.5,dashdotted,mark options={ fill=Blue}]table[x index = 5, y index = 6] {data_rev/MNIST_DX_time_res_maxit50_tutti};
		   \addplot+ [mark=none,Peach,line width=1.8,mark options={ fill=Peach}]table[x index = 7, y index = 8] {data_rev/MNIST_DX_time_res_maxit50_tutti};
		\end{axis}
\end{tikzpicture}%
\begin{tikzpicture}
		\begin{axis}[width=. 5* \textwidth, 
		   title = {\small Fashion MNIST}, 
		   xlabel = {\small  CPU time (secs)}, 
y tick label style={/pgf/number format/sci},
		   height = .2 \textheight,
		   		   xmin = -.1,
		 ]
\addplot+[mark=none,OrangeRed,line width=1.5,loosely dotted,mark options={ fill=OrangeRed}] table[x index = 1, y index = 2] {data_rev/FASHION_DX_time_res_maxit50_tutti};
        \addplot+ [mark=none,PineGreen,line width=1.5,dashed,mark options={ fill=PineGreen}]table[x index = 3, y index = 4] {data_rev/FASHION_DX_time_res_maxit50_tutti};
        \addplot+ [mark=none,Blue,line width=1.5,dashdotted,mark options={ fill=Blue}]table[x index = 5, y index = 6] {data_rev/FASHION_DX_time_res_maxit50_tutti};
		   \addplot+ [mark=none,Peach,line width=1.8,mark options={ fill=Peach}]table[x index = 7, y index = 8] {data_rev/FASHION_DX_time_res_maxit50_tutti};

		\end{axis}
\end{tikzpicture}
\\
\ref{named9}
\caption{Residual norm history for PALM-DL, iPALM-DL, \bipalmdot-DL, \spalmdot-DL
(50 iterations). \label{fig:res_bipalm_maxit}}

\end{figure}
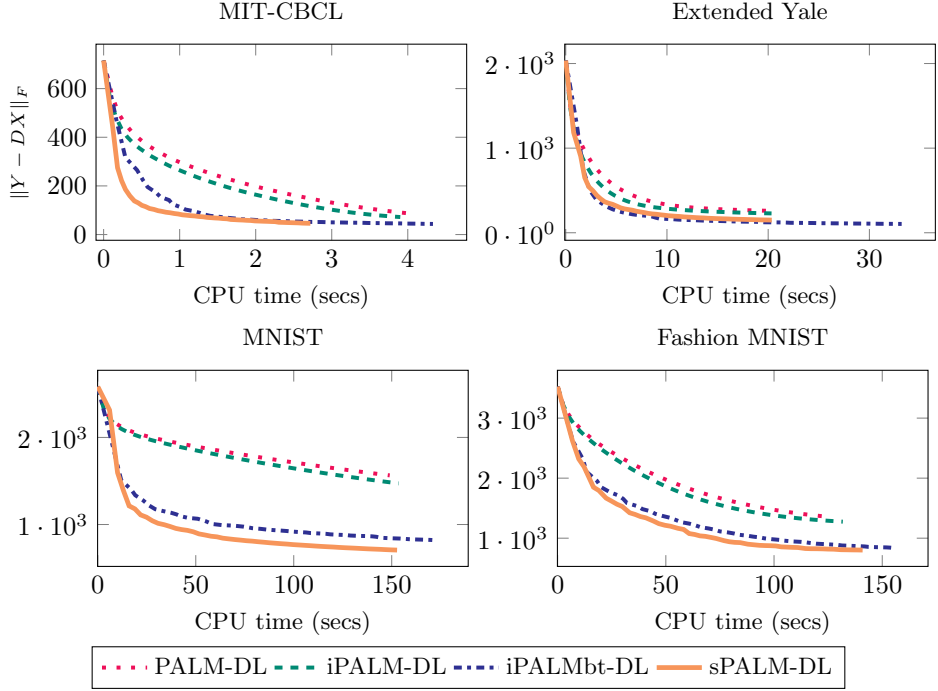


{
For the sake of completeness we also report in Table~\ref{tab:comparison_ksvd} a comparison between the
classical K-SVD method  {\cite{Aharon2006}}\footnote{We used the Matlab implementation KSVD-Box v13 of K-SVD available at  http://www.cs.technion.ac.il/~ronrubin/software.html}} and the sPALM-DL algorithm in terms of CPU time needed to obtain a comparable 
value of the objective function. 
For this experiment,  sPALM was run until the objective function value
was smaller than that of the objective function obtained with two iterations of K-SVD. Notice that for all the databases, except MNIST, sPALM-DL takes less CPU time than K-SVD, although K-SVD} {uses optimized MEX functions written in C.
Focusing on MNIST, we observe that sPALM takes 84.7 seconds to reach the value of the residual obtained by K-SVD within the 5\%.}

\begin{table}[hbt]
\centering{
\begin{tabular}{lrrrr}
\toprule
          &  \multicolumn{2}{c}{sPALM-DL}  & \multicolumn{2}{c}{K-SVD}   \\ \midrule
&  \small{$\|Y-DX\|_F$} & \small{CPU time} &  \small{$\|Y-DX\|_F$} & \small{CPU time}   \\ 
\midrule
 MIT-CBCL & $61.9$ & $1.7$ & $62.7$ & $8.6$ \\ 
Ext'd Yale shrunk & $109.7$ & $77.7$ & $109.8$ & $148.3$  \\
MNIST & $761.1$ & $105.4$ & $763.1$ & $91.3$ \\ 
Fashion MNIST & $827.7$ & $114.4$ & $829.5$ & $140.9$ \\ 
\bottomrule
\end{tabular}
\caption{Value of the objective function and CPU time  memory requirements for  sPALM-DL and K-SVD.}
\label{tab:comparison_ksvd}}
\end{table}


\subsection{Matrix vs tensor  DL classification problem}
Given the training set $\mathcal{Y}\in\mathbb{R}^{n_1\times n_2\times \tilde{n}_e\times n_p}$, we 
{now solve the DL classification problem  using either a matrix or a tensor formulation.
The matrix DL problem (\ref{eq:dl_ind}) is solved using 
PALM-DL and \spalmdot-DL. The tensor problem (\ref{eq:PALM_tt})
is solved by PALM-DL-TT and {\spalmttdot}. In all cases, the classification matrix $W$ is then computed by solving (\ref{eq:W}). 

Figure~\ref{fig:class_noW} displays the classification success rate of all algorithms as the iterations proceed.
The use of a spectral step results in higher 
classification performance for all examined data. For the MIT-CBCL, \spalm based algorithms achieve 
the maximum classification rate after 20 iterations while PALM-DL and PALM-DL-TT need more iterations
to reach the same rate. When processing MNIST the classification performance of PALM-DL and PALM-DL-TT decreases 
as iterations progress, suggesting overfitting, whereas a slight improvement
occurs with \spalmdot-DL and \spalmttdot.
On these datasets, the Tensor-Train formulation does not seem to be beneficial for classification purposes. 
Computer memory limitations however may favour the tensor approach, as we will discuss in
the next section.

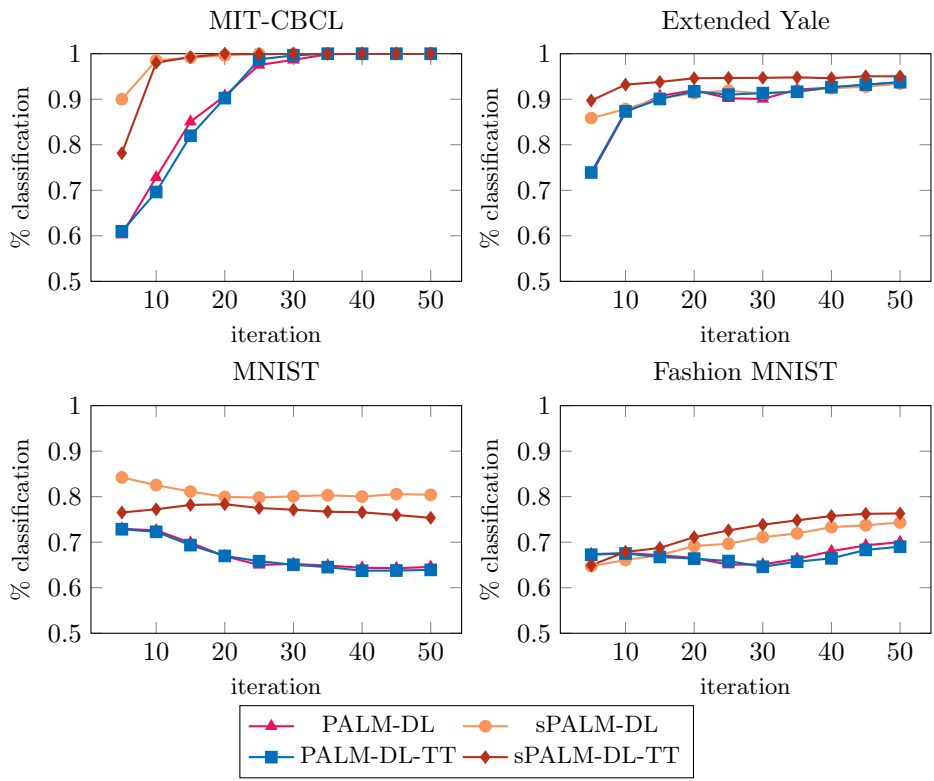
\begin{figure} [!ht]
\centering
\begin{tikzpicture}
		\begin{axis}[width=. 5* \textwidth, 
		   title = {MIT-CBCL}, 
		   legend columns=2,
           legend entries={ \small PALM-DL, \small \spalmdot-DL, \small PALM-DL-TT,  \small \spalmtt },
           legend to name=named2,
		   legend cell align=center,	
		   xlabel = {\small iteration}, 
		   ylabel = {\small $\%$ classification},
		   height = .22 \textheight,
           ytick = {0.5,0.6,...,1},
           ymax = 1,
           ymin = 0.5
		 ]
		  \addplot+[mark=triangle*,OrangeRed,thick,mark options={ fill=OrangeRed}] table[x index = 0, y index = 1] {data_2306/data_classification_MIT_noW};
		   \addplot+ [mark=*,Peach,thick,mark options={ fill=Peach}]table[x index = 0, y index = 2] {data_2306/data_classification_MIT_noW};
		    \addplot+ [mark=square*,NavyBlue,thick,mark options={ fill=NavyBlue}]table[x index = 0, y index = 3] {data_2306/data_classification_MIT_noW};
		    \addplot+ [mark=diamond*, 	BrickRed,thick,mark options={ fill= 	BrickRed}]table[x index = 0, y index = 4] {data_2306/data_classification_MIT_noW};
		\end{axis}
\end{tikzpicture}
~\begin{tikzpicture}
		\begin{axis}[width=. 5* \textwidth, 
		   title = {Extended Yale}, 
		   xlabel = {\small iteration}, 
		   ylabel = {\small $\%$ classification},
		   height = .22 \textheight,
           ytick = {0.5,0.6,...,1},
           ymax = 1,
           ymin = 0.5
		 ]
			  \addplot+[mark=triangle*,OrangeRed,thick,mark options={ fill=OrangeRed}] table[x index = 0, y index = 1] {data_2306/data_classification_exyale_noW};
		   \addplot+ [mark=*,Peach,thick,mark options={ fill=Peach}]table[x index = 0, y index = 2] {data_2306/data_classification_exyale_noW};
		    \addplot+ [mark=square*,NavyBlue,thick,mark options={ fill=NavyBlue}]table[x index = 0, y index = 3] {data_2306/data_classification_exyale_noW};
		    \addplot+ [mark=diamond*,BrickRed,thick,mark options={ fill=BrickRed}]table[x index = 0, y index = 4] {data_2306/data_classification_exyale_noW};
		\end{axis}
\end{tikzpicture}
\\
\begin{tikzpicture}
 		\begin{axis}[width=. 5* \textwidth, 
 		   title = {MNIST}, 
           xlabel = {\small iteration}, 
		   ylabel = {\small $\%$ classification},
 		   height = .22 \textheight,
            ytick = {0.5,0.6,...,1},
            ymax = 1,
            ymin = 0.5
 		 ]
  		  \addplot+[mark=triangle*,OrangeRed,thick,mark options={ fill=OrangeRed}] table[x index = 0, y index = 1] {data_2306/data_classification_MNIST_noW};
 		   \addplot+ [mark=*,Peach,thick,mark options={ fill=Peach}]table[x index = 0, y index = 2] {data_2306/data_classification_MNIST_noW};
 		    \addplot+ [mark=square*,NavyBlue,thick,mark options={ fill=NavyBlue}]table[x index = 0, y index = 3] {data_2306/data_classification_MNIST_noW};
 		    \addplot+ [mark=diamond*,BrickRed,thick,mark options={ fill=BrickRed}]table[x index = 0, y index = 4] {data_2306/data_classification_MNIST_noW};
 		 		\end{axis}
 \end{tikzpicture}
  ~\begin{tikzpicture}
 		\begin{axis}[width=. 5* \textwidth, 
 		   title = {Fashion MNIST}, 
           xlabel = {\small iteration}, 
		   ylabel = {\small $\%$ classification},
 		   height = .22 \textheight,
           ytick = {0.5,0.6,...,1},
           ymax = 1,
           ymin = 0.5
 		 ]
 		  \addplot+[mark=triangle*,OrangeRed,thick,mark options={ fill=OrangeRed}] table[x index = 0, y index = 1] {data_2306/data_classification_fashionMNIST_noW};
 		   \addplot+ [mark=*,Peach,thick,mark options={ fill=Peach}]table[x index = 0, y index = 2] {data_2306/data_classification_fashionMNIST_noW};
 		    \addplot+ [mark=square*,NavyBlue,thick,mark options={ fill=NavyBlue}]table[x index = 0, y index = 3] {data_2306/data_classification_fashionMNIST_noW};
 		    \addplot+ [mark=diamond*,BrickRed,thick,mark options={ fill=BrickRed}]table[x index = 0, y index = 4] {data_2306/data_classification_fashionMNIST_noW};
 		\end{axis}
 \end{tikzpicture}
\\
\ref{named2}
\caption{Classification performance of PALM-DL, \spalmdot-DL, PALM-DL-TT and \spalmtt with respect to the number of iterations for four different database, using the formulations of Section \ref{sec:dl_problem}.}
\label{fig:class_noW}
\end{figure}


%

\subsection{Memory saving truncated approach}
One of the challenges in dealing with huge databases is to reduce memory requirements. 
For a database $\mathcal{Y}\in\mathbb{R}^{n_1\times n_2 \times n_e \times n_p}$, PALM-DL and
sPALM-DL store 
$\mathcal{D}\in\mathbb{R}^{n_1\times n_2\times k}$ and 
(sparse) $\mathcal{X}\in\mathbb{R}^{k\times n_e \times n_p}$, requiring 
$m_P:=n_1n_2k+\tau n_e n_p$  memory allocations.
This quantity can be quite large in real image applications. 
We next investigate the possibility of truncating the tensor decomposition, possibly without interfering with
the classification performance.
In the Tensor-Train based algorithms PALM-DL-TT and \spalmtt\, storage for the arrays
$G_1\in\mathbb{R}^{n_1\times r_1}$, 
$\mathcal{G}_2\in\mathbb{R}^{r_1\times n_2\times r_2}$, 
$G_3\in\mathbb{R}^{r_2\times k}$ and 
$\mathcal{X}\in\mathbb{R}^{k\times n_e \times n_p}$ is required, yielding
$m_{TT}:=n_1r_1+r_1n_2r_2+r_2k+\tau n_en_p$ allocations. 
The value of $r_2$ determines whether the (truncated) TT approach is more memory efficient than
the full scheme by comparing  $m_{TT}$ and $m_P$. 
In Figure \ref{fig:class_PALM_r2} we show the classification rates for all TT based methods on two
of the datasets after 50 iterations, as $r_2$ varies up to the maximum value obtainable for that dataset ($r_2\le 225$ and
$r\le 300$ for MIT-CBCL and Extended Yale, resp.).
%
We note that {the TT variants} are able to achieve good classification performance also with small values of $r_2$. 
In particular, for the MIT-CBCL and \spalmtt choosing a value of $r_2$ greater than $40$ has no benefit on the 
classification performance, suggesting the use of $r_2=40$, thus reducing the overall memory costs
with respect to PALM ($m_P=108,945$ vs $m_{TT}=36,585$). 
Similarly, for Extended Yale the {value $r_2=150$ } can be chosen without 
dramatically spoiling the classification performance. 
In other words the Tensor-Train Decomposition enables us to store the information 
for classification purposes in a more compact manner. 

\begin{figure}[!ht]
\centering
\begin{tikzpicture}
		\begin{axis}[width=. 5* \textwidth, 
		   title = {MIT-CBCL}, 
		   legend columns=4,
           legend entries={ \small PALM-DL-TT, \small \spalmtt, \small PALM-CDL-TT,  \small sPALM-CDL-TT },
           legend to name=named4,
		   legend cell align=center,
		   xlabel = {\small TT-rank $r_2$}, 
		   ylabel = {\small $\%$ classification},
		   height = .22 \textheight,
		 ]
		  \addplot+[mark=triangle*,OrangeRed,thick,mark options={ fill=OrangeRed}] table[x index = 0, y index = 1] {data_2306/data_truncation_MIT};
		   \addplot+ [mark=*,Peach,thick,mark options={ fill=Peach}]table[x index = 0, y index = 2] {data_2306/data_truncation_MIT};
		    \addplot+ [mark=square*,NavyBlue,thick,mark options={ fill=NavyBlue}]table[x index = 0, y index = 3] {data_2306/data_truncation_MIT};
		    \addplot+ [mark=diamond*,BrickRed,thick,mark options={ fill=BrickRed}]table[x index = 0, y index = 4] {data_2306/data_truncation_MIT};
		\end{axis}
\end{tikzpicture}
\begin{tikzpicture}
		\begin{axis}[width=. 5* \textwidth, 
		   title = {Extended Yale}, 
		   xlabel = {\small TT-rank $r_2$}, 
		   ylabel = {\small $\%$ classification},
		   height = .22 \textheight,
		 ]
		  \addplot+[mark=triangle*,OrangeRed,thick,mark options={ fill=OrangeRed}] table[x index = 0, y index = 1] {data_2306/data_truncation_EXYALE};
		   \addplot+ [mark=*,Peach,thick,mark options={ fill=Peach}]table[x index = 0, y index = 2] {data_2306/data_truncation_EXYALE};
		    \addplot+ [mark=square*,NavyBlue,thick,mark options={ fill=NavyBlue}]table[x index = 0, y index = 3] {data_2306/data_truncation_EXYALE};
		    \addplot+ [mark=diamond*,BrickRed,thick,mark options={ fill=BrickRed}]table[x index = 0, y index = 4] {data_2306/data_truncation_EXYALE};
		\end{axis}
 \end{tikzpicture}
\\
\ref{named4}
\caption{Classification performance of the Tensor-Train based algorithms for different values of the TT-rank $r_2$.}
\label{fig:class_PALM_r2}
\end{figure}
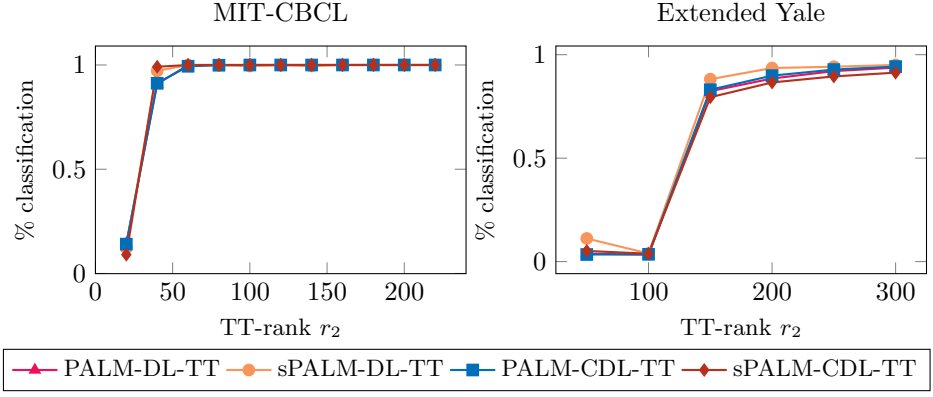

\subsection{A classification example in 4D setting}\label{sec:4D}
The Tensor-Train decomposition allows us to
readily extend the 3D formulation, explored in the previous sections, to higher-order tensors. In the following 
we analyze the classification performance of \palmtt and \spalmtt for a $5$th order tensor 
$\mathcal{Y}\in\mathbb{R}^{n_1\times n_2 \times n_3\times \tilde{n}_e\times n_p}$ and we compare them with their matrix versions PALM-DL and \spalmdot-DL. First of all, we write the TT dictionary learning problem as in 
(\ref{eq:dl_tensor_N}) with $q=3$ and $s=5$. 
Furthermore, we notice that
 
{\footnotesize
 \begin{eqnarray} \label{eq:ttdl_4d}
\hskip -.3in\|\mathcal{Y}-\left(G_1\times_2^1 \mathcal{G}_2 \times_3^1\mathcal{G}_3 \times_3^1 G_4\right)
\times_{4}^1\mathcal{X}\|_F
=
\|\mathcal{Y}_{[3]}-\left(I_{n_2n_3}\otimes G_1\right)\left(I_{n_3}\otimes\left(\mathcal{G}_2\right)_{[2]}\right)
\left(\mathcal{G}_3\right)_{[2]}G_4\mathcal{X}_{[1]}\|_F
\end{eqnarray}
}
where $G_1$,$\left(\mathcal{G}_2\right)_{[2]}$, $\left(\mathcal{G}_3\right)_{[2]}$, are matrices with 
orthonormal columns, and more precisely,
$G_1\in\Theta_{n_1,r_1}$, $\left(\mathcal{G}_2\right)_{[2]}\in\Theta_{r_1n_2,r_2}$, 
$\left(\mathcal{G}_3\right)_{[2]}\in\Theta_{r_2n_3,r_3}$ 
while $G_4\in\Omega_{r_3,k}$ has unit norm columns. 
The following proposition provides an expression for the gradient of $H$ and 
corresponding Lipschitz constants using the orthogonality of the first three TT-cores.
 
\vskip 0.1in
 \begin{proposition} \label{prop:lipschitz_PALMTT_4D}
Let $G_2=\left(\mathcal{G}_2\right)_{[2]}$, 
$G_3=\left(\mathcal{G}_3\right)_{[2]}$,
$X=\mathcal{X}_{[1]}$, $Y=\mathcal{Y}_{[3]}$ and 
  \begin{equation*}
  H(G_1, G_2 ,G_3,G_4, X)=\left\|Y-\left(I_{n_2n_3}\otimes G_1\right)\left(I_{n_3}\otimes G_2\right)G_3 G_4X\right\|_F^2.
  \end{equation*}
 Then the partial gradients of $H$ satisfy Assumption~A3. 
Moreover, the following upper bounds for the Lipschitz constants hold: 
$L_X = 2 \|G_4\|_2^2$, $L_{G_3}= 2 \|G_4 X \|_2^2$, $L_{G_4}= 2 \|X \|_2^2$,
$L_{G_1} = 2  \|\sum_{i=1}^p A_i A_i^T\|_2$, $L_{G_2}= 2 \|\sum_{i=1}^p B_i B_i^T\|_2$,
where $A_i\in\mathbb{R}^{r_1\times n_2}$ is the matricization of the $i$th column of  $A=(I_{n_3}\otimes G_2)\, G_3 G_4 X$   and $B_i\in\mathbb{R}^{r_1\times n_2}$ is the matricization of the $i$th column of $B=G_3 G_4X$.
 \end{proposition}
 \begin{proof}
  By direct computation we obtain the following expressions for the partial gradients of $H$:
\begin{eqnarray*}
   \nabla_X H = 
-2 \left(I_{n_2n_3}\otimes G_1\right)\left(I_{n_3}\otimes G_2\right)G_3 G_4)^T Y +2 (G_3 G_4)^T G_3 G_4X 
%
%
\end{eqnarray*}
and $\nabla_{G_1} H=2 \sum_{i=1}^p (-Y_i + G_1A_i)A_i^T$, 
where $Y_i$  denotes the matricization of the $i$th column of $Y$, 
$\nabla_{G_2} H= {2 \sum_{i=1}^p (-Y_i + G_1B_i)B_i^T}$, 
%
%
\begin{eqnarray*}
    \nabla_{G_3}H = 2((I_{n_2n_3}\otimes G_1)(I_{n_3}\otimes G_2))^T
\left ( - Y + (I_{n_2n_3}\otimes G_1)(I_{n_3}\otimes G_2)) G_3 G_4 X\right )  (G_4 X)^T  .
  \end{eqnarray*}
  For $\nabla_{G_4}$, we have $\nabla_{G_4}H=2 C^T( -Y+C G_4X) X^T$,  where 
$C=(I_{n_2n_3}\otimes G_1 ) (I_{n_3}\otimes G_2) G_3$.
  Using these expressions we follow the proof of Proposition~\ref{prop:lipschitz_PALMTT} to obtain the 
required Lipschitz moduli.
 \end{proof}

 To test this formulation we consider the COIL-100 database \cite{CAVE_0189}, containing
RGB images of $n_p=100$ different objects in $n_e=72$ view angles. 
For this experiment the size of each image is reduced to $16\times 16\times 3$ to 
preserve the relation among the number of pixels $n=n_1n_2n_3$, the atoms of the dictionary 
$k$ and the total number of images $n_en_p$ (i.e., $n<k<n_en_p$). The 
training set $\mathcal{Y}$ is composed by all the objects in ${\tilde n}_e = 54$ different view angles, 
 corresponding to $75\%$ of the total number of images. 
We set the number of atoms to $k=1400$ and the sparsity parameter to $\tau=4n_p$. 
The reported results correspond to the classification success along the iterations for a maximum of 50 iterations.
%
From Figure~\ref{fig:class_4d} we can notice that the spectral step enhances the 
classification rate  significantly for the TT formulation. In particular at 
the 50$th$ iteration the classification rate of \spalmtt is equal to 
$77.11\%$ while the other methods do not reach $70\%$. Furthermore we observe that  using a number of 
iterations greater than $30$ has almost no impact on the classification rate. 

\begin{figure}
\centering
\begin{tikzpicture}
		\begin{axis}[width=. 8* \textwidth, 
		   title = {COIL-100}, 
		   legend columns=4,
           legend entries={ \small PALM-DL, \small \spalmdot-DL, \small PALM-DL-TT,  \small \spalmtt },
           legend to name=named7,
		   legend cell align=center,
		   xlabel = {\small iteration}, 
		   ylabel = {\small $\%$ classification},
		   height = .22 \textheight,
           ytick = {0.5,0.6,...,0.8},
           ymax = 0.8,
           ymin = 0.5
		 ]
		  \addplot+[mark=triangle*,OrangeRed,thick,mark options={ fill=OrangeRed}] table[x index = 0, y index = 1] {data_2306/data_classification_4d};
		   \addplot+ [mark=*,Peach,thick,mark options={ fill=Peach}]table[x index = 0, y index = 2] {data_2306/data_classification_4d};
		    \addplot+ [mark=square*,NavyBlue,thick,mark options={ fill=NavyBlue}]table[x index = 0, y index = 3] {data_2306/data_classification_4d};
		    \addplot+ [mark=diamond*, 	BrickRed,thick,mark options={ fill= 	BrickRed}]table[x index = 0, y index = 4] {data_2306/data_classification_4d};
		\legend{ \small PALM-{DL}, \small \spalmdot-{DL}, \small PALM-{DL}-TT,  \small \spalmttdot  };
		\end{axis}
\end{tikzpicture}
\\
\ref{named7}
\caption{Classification performance of PALM-DL, \spalmdot-DL, PALM-DL-TT and 
\spalmtt with respect to the number of iterations for the COIL-100 database.\label{fig:class_4d}}
\end{figure}
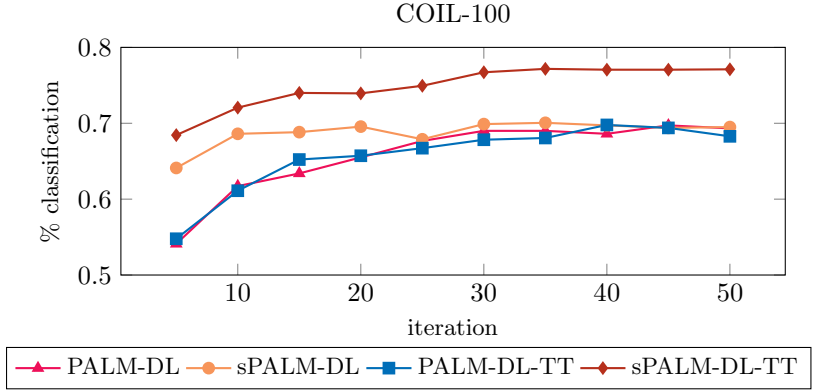

\section{Conclusions} \label{sec:conclusions}
Exploiting  data multidimensionality is crucial in dictionary learning and, in particular, when DL 
is applied to image classification.
To address this challenge, we have devised the new  method \spalmdot\ in the class of PALM-type convergent algorithms, and
proposed both a matrix and tensor-train formulation for its implementation.
The new approach implicitly includes second order information by using
spectral stepsizes {that exploit the alternating iteration history}.
The resulting algorithm is competitive with respect to different recent variants of
PALM, as illustrated by our numerical experiments. Moreover, 
since sPALM is described for general nonconvex non-smooth 
problems, it may serve as a basis for further algorithmic acceleration 
techniques \cite{Bonettini2,TITAN,iPALM}. 
Finally, we have experimentally shown that the TT formulation may bring advantages in terms of memory
requirements and rate of successful classification, especially when applied to 4D databases.

\section*{Acknowledgments}
We would like to thank Nicolas Gillis for insightful comments on a previous version of this work.

\section*{Data availability} The data that support the findings of this study are available from the corresponding author upon request.

\bibliography{./bibliografia_jabref}

\appendix


\section{Analysis of the convergence history of the PALM variants in solving the DL problem} \label{app:test}{
In this section we compare the matrix based methods in Table \ref{codes}
for the solution of the matrix DL problem (\ref{eq:dl_matrix}).
To this purpose, in addition to a safeguard
strategy on a maximum (loose) number of iterations,
we consider the following stopping criterion based on iterate variation, that 
takes into account possible different scalings in the block variables. The criterion is given by
$
\|\bar D -D\|_F\leq {\tt tol_D},\ \|\bar X - X\|_F  \leq {\tt tol_X}
$. 
Here the upper bar denotes the approximate solution from the previous iteration
and  ${\tt tol_D}$ and ${\tt tol_X}$ are tolerances set equal to ${\tt tol_D}=10^{-3}\sqrt{n_1n_2k}$ and ${\tt tol_X}=10^{-3}\sqrt{kn_en_p}$.

The plots in Figure~\ref{fig:res_bipalm} show the values of 
$H(D,X)=\|Y-DX\|_F$ as the CPU time proceeds, for the datasets in Table~\ref{tab:car_data}. 
As expected, the use of a backtracking rule (iPALMbt-DL) to estimate the Lipschitz constants  yields a much faster decrease in the residual value than using the constant stepsize (PALM-DL and iPALM-DL) but  each iteration of iPALMbt-DL is more expensive,
see the  steeper slope in Figure~\ref{fig:res_bipalm_iter_cpu}.
Moreover, for all the databases, \spalmdot-DL converges in far fewer iterations than 
\bipalmdot-DL, resulting in a significantly lower overall CPU time than for the other methods. These
results illustrate the advantage of using higher order information.  

}

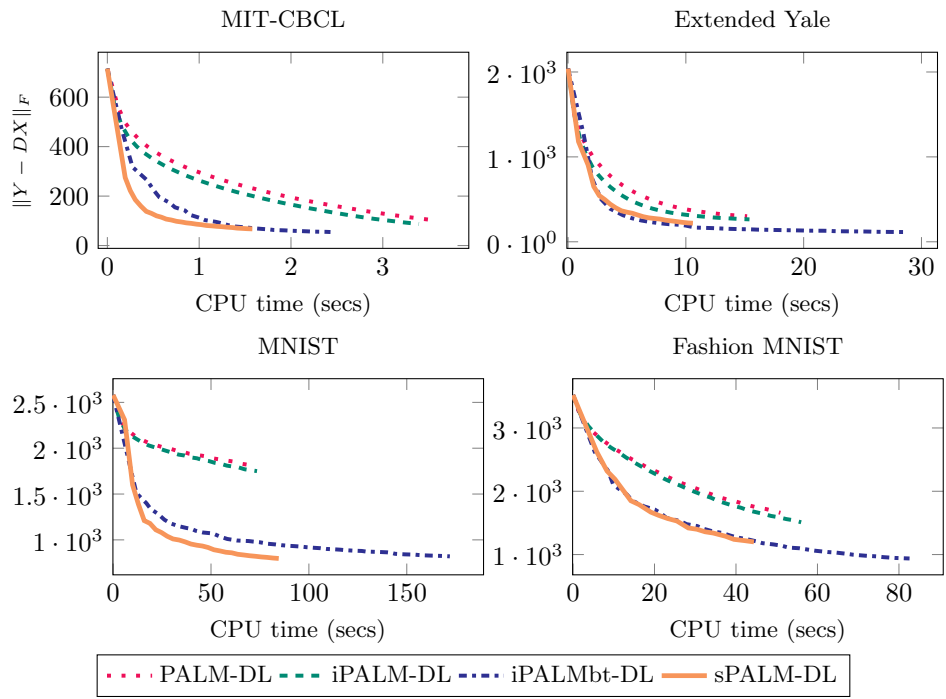
\begin{figure}[htb]
\centering
\begin{tikzpicture}
		\begin{axis}[width=.5* \textwidth, 
		   title = {\small MIT-CBCL}, 
		   legend columns=4,
           legend entries={ \small PALM-DL,  \small iPALM-DL, \small \bipalmdot-DL, \small \spalmdot-DL },
           legend to name=named,
		   xlabel = {\small  CPU time (secs)}, 
		   ylabel = {\scriptsize $\|Y-DX\|_F$},
		   height = .2 \textheight,
		   xmin = -.1,
		 ]
		  \addplot+[mark=none,OrangeRed,line width=1.5,loosely dotted,mark options={ fill=OrangeRed}] table[x index = 1, y index = 2] {data_rev/MIT_DX_time_iter_palm};
        \addplot+ [mark=none,PineGreen,line width=1.5,dashed,mark options={ fill=PineGreen}]table[x index = 1, y index = 2] {data_rev/MIT_DX_time_iter_ipalm};
        \addplot+ [mark=none,Blue,line width=1.5,dashdotted,mark options={ fill=Blue}]table[x index = 1, y index = 2] {data_rev/MIT_DX_time_iter_ipalmbt_spalm};
		   \addplot+ [mark=none,Peach,line width=1.8,mark options={ fill=Peach}]table[x index = 3, y index = 4] {data_rev/MIT_DX_time_iter_ipalmbt_spalm};	    
		\end{axis}
\end{tikzpicture}	
\begin{tikzpicture}
		\begin{axis}[width=.5* \textwidth, 
		   title = {\small Extended Yale}, 
		    xlabel = {\small  CPU time (secs)},  
y tick label style={/pgf/number format/sci},
		   height = .2 \textheight,
		   		   xmin = -.1,
		 ]
		 \addplot+[mark=none,OrangeRed,line width=1.5,loosely dotted,mark options={ fill=OrangeRed}] table[x index = 1, y index = 2] {data_rev/YALE_DX_time_iter_palm_ipalm};
        \addplot+ [mark=none,PineGreen,line width=1.5,dashed,mark options={ fill=PineGreen}]table[x index = 3, y index = 4] {data_rev/YALE_DX_time_iter_palm_ipalm};
        \addplot+ [mark=none,Blue,line width=1.5,dashdotted,mark options={ fill=Blue}]table[x index = 1, y index = 2] {data_rev/YALE_DX_time_iter_ipalmbt};
		   \addplot+ [mark=none,Peach,line width=1.8,mark options={ fill=Peach}]table[x index = 1, y index = 2] {data_rev/YALE_DX_time_iter_spalm};
		\end{axis}
\end{tikzpicture}
\begin{tikzpicture}
		\begin{axis}[width=.5* \textwidth, 
		   title = {\small MNIST}, 
		   xlabel = {\small  CPU time (secs)},  
y tick label style={/pgf/number format/sci},
		   height = .2 \textheight,
		   		   xmin = -.1,
		 ]
		 \addplot+[mark=none,OrangeRed,line width=1.5,loosely dotted,mark options={ fill=OrangeRed}] table[x index = 1, y index = 2] {data_rev/MNIST_DX_time_iter_palm_ipalm};
        \addplot+ [mark=none,PineGreen,line width=1.5,dashed,mark options={ fill=PineGreen}]table[x index = 3, y index = 4] {data_rev/MNIST_DX_time_iter_palm_ipalm};
        \addplot+ [mark=none,Blue,line width=1.5,dashdotted,mark options={ fill=Blue}]table[x index = 1, y index = 2] {data_rev/MNIST_DX_time_iter_ipalmbt};
		   \addplot+ [mark=none,Peach,line width=1.8,mark options={ fill=Peach}]table[x index = 1, y index = 2] {data_rev/MNIST_DX_time_iter_spalm};
		\end{axis}
\end{tikzpicture}%
\begin{tikzpicture}
		\begin{axis}[width=. 5* \textwidth, 
		   title = {\small Fashion MNIST}, 
		   xlabel = {\small  CPU time (secs)}, 
y tick label style={/pgf/number format/sci},
		   height = .2 \textheight,
		   		   xmin = -.1,
		 ]
\addplot+[mark=none,OrangeRed,line width=1.5,loosely dotted,mark options={ fill=OrangeRed}] table[x index = 1, y index = 2] {data_rev/FASHION_DX_time_iter_palm_ipalm};
        \addplot+ [mark=none,PineGreen,line width=1.5,dashed,mark options={ fill=PineGreen}]table[x index = 3, y index = 4] {data_rev/FASHION_DX_time_iter_palm_ipalm};
        \addplot+ [mark=none,Blue,line width=1.5,dashdotted,mark options={ fill=Blue}]table[x index = 1, y index = 2] {data_rev/FASHION_DX_time_iter_ipalmbt};
		   \addplot+ [mark=none,Peach,line width=1.8,mark options={ fill=Peach}]table[x index = 1, y index = 2] {data_rev/FASHION_DX_time_iter_spalm};

		\end{axis}
\end{tikzpicture}
\\
\ref{named}
\caption{Residual norm history for PALM-DL, iPALM-DL, \bipalmdot-DL, \spalmdot-DL. \label{fig:res_bipalm}}
\end{figure}
 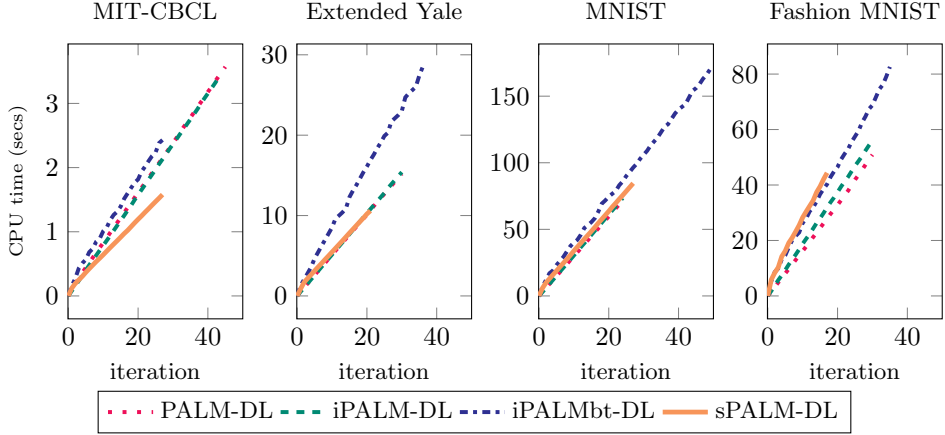
\begin{figure}[htb]
 \centering
 \begin{tikzpicture}
 		\begin{axis}[width=.3* \textwidth, 
 		   title = {\small MIT-CBCL}, 
 		   legend columns=4,
            legend entries={ \small PALM-DL, \small iPALM-DL, \small \bipalmdot-DL, \small \spalmdot-DL },
            legend to name=named5,
 		   xlabel = {\small  iteration}, 
 	   ylabel = { \scriptsize CPU time (secs)}, 
 		   height = .25 \textheight,
 		   		   xmin = -.1,
            xmax = 50
 		 ]
 		  \addplot+[mark=none,OrangeRed,line width=1.5,loosely dotted,mark options={ fill=OrangeRed}] table[x index = 0, y index = 1] {data_rev/MIT_DX_time_iter_palm};
        \addplot+ [mark=none,PineGreen,line width=1.5,dashed,mark options={ fill=PineGreen}]table[x index = 0, y index = 1] {data_rev/MIT_DX_time_iter_ipalm};
        \addplot+ [mark=none,Blue,line width=1.5,dashdotted,mark options={ fill=Blue}]table[x index = 0, y index =1] {data_rev/MIT_DX_time_iter_ipalmbt_spalm};
		   \addplot+ [mark=none,Peach,line width=1.8,mark options={ fill=Peach}]table[x index = 0, y index = 3] {data_rev/MIT_DX_time_iter_ipalmbt_spalm};	  
 		\end{axis}
 \end{tikzpicture}
 \begin{tikzpicture}
 		\begin{axis}[width=.3* \textwidth, 
 		   title = {\small Extended Yale}, 
 		   xlabel = {\small  iteration}, 
 		   height = .25 \textheight,
 		   		   xmin = -.1,
            xmax = 50
 		 ]
 \addplot+[mark=none,OrangeRed,line width=1.5,loosely dotted,mark options={ fill=OrangeRed}] table[x index = 0, y index = 1] {data_rev/YALE_DX_time_iter_palm_ipalm};
        \addplot+ [mark=none,PineGreen,line width=1.5,dashed,mark options={ fill=PineGreen}]table[x index = 0, y index = 3] {data_rev/YALE_DX_time_iter_palm_ipalm};
        \addplot+ [mark=none,Blue,line width=1.5,dashdotted,mark options={ fill=Blue}]table[x index = 0, y index = 1] {data_rev/YALE_DX_time_iter_ipalmbt};
		   \addplot+ [mark=none,Peach,line width=1.8,mark options={ fill=Peach}]table[x index = 0, y index = 1] {data_rev/YALE_DX_time_iter_spalm};
 		\end{axis}
 \end{tikzpicture}
 \begin{tikzpicture}
 		\begin{axis}[width=.3* \textwidth, 
 		   title = {\small MNIST}, 
 		   xlabel = {\small  iteration}, 
 		   height = .25 \textheight,
 		   		   xmin = -.1,
            xmax = 50
 		 ]
 		 \addplot+[mark=none,OrangeRed,line width=1.5,loosely dotted,mark options={ fill=OrangeRed}] table[x index = 0, y index = 1] {data_rev/MNIST_DX_time_iter_palm_ipalm};
        \addplot+ [mark=none,PineGreen,line width=1.5,dashed,mark options={ fill=PineGreen}]table[x index = 0, y index = 3] {data_rev/MNIST_DX_time_iter_palm_ipalm};
        \addplot+ [mark=none,Blue,line width=1.5,dashdotted,mark options={ fill=Blue}]table[x index = 0, y index = 1] {data_rev/MNIST_DX_time_iter_ipalmbt};
		   \addplot+ [mark=none,Peach,line width=1.8,mark options={ fill=Peach}]table[x index = 0, y index = 1] {data_rev/MNIST_DX_time_iter_spalm};
 		\end{axis}
 \end{tikzpicture}
 \begin{tikzpicture}
 		\begin{axis}[width=.3* \textwidth, 
 		   title = {\small Fashion MNIST}, 
 		   xlabel = {\small  iteration}, 
 		   height = .25 \textheight,
 		   		   xmin = -.1,
            xmax = 50
 		 ]
\addplot+[mark=none,OrangeRed,line width=1.5,loosely dotted,mark options={ fill=OrangeRed}] table[x index = 0, y index = 1] {data_rev/FASHION_DX_time_iter_palm_ipalm};
        \addplot+ [mark=none,PineGreen,line width=1.5,dashed,mark options={ fill=PineGreen}]table[x index = 0, y index = 3] {data_rev/FASHION_DX_time_iter_palm_ipalm};
        \addplot+ [mark=none,Blue,line width=1.5,dashdotted,mark options={ fill=Blue}]table[x index = 0, y index = 1] {data_rev/FASHION_DX_time_iter_ipalmbt};
		   \addplot+ [mark=none,Peach,line width=1.8,mark options={ fill=Peach}]table[x index = 0, y index = 1] {data_rev/FASHION_DX_time_iter_spalm};
 		\end{axis}
 \end{tikzpicture}
 \\
 \ref{named5}
 \caption{CPU time as iterations increase, for PALM-DL,  iPALM-DL, \bipalmdot-DL, \spalmdot-DL. \label{fig:res_bipalm_iter_cpu}}
 \end{figure}



\section{Tools for Tensor-Train Decomposition}
\label{sec:tt_tools}
In this section we introduce some tensor notation and useful tools.
The \emph{fibers} of a tensor are obtained by fixing every index but one. 
In particular for third order tensors we can define column fibers (last two indices fixed), row fibers (first and third indices fixed) and tube fibers (first two indices fixed). The \emph{slices} of a tensor are defined by fixing only one index. For third order tensors, we can define the horizontal slices (first index fixed), the lateral slices (second index fixed) and the frontal slices (third index fixed).

\begin{definition} {\rm \cite[p.~458]{kolda}}
\label{def:scal_prod}
 Given two tensors $\mathcal{A}\in\mathbb{R}^{I_1\times\dots\times I_N}$ and $\mathcal{B}\in\mathbb{R}^{I_1\times\dots\times I_N}$, the scalar product between $\mathcal{A}$ and $\mathcal{B}$ is defined as
 \begin{equation}
  \langle \mathcal{A},\mathcal{B}\rangle=\sum_{i_N=1}^{I_N}\dots \sum_{i_1=1}^{I_1} a_{i_1,\dots, i_N}b_{i_1,\dots,i_N}.
 \end{equation}
\end{definition}
The Frobenius norm of a tensor is given 
by $\|\mathcal{A}\|^2_F=\langle\mathcal{A},\mathcal{A}\rangle$. 
If $\mathcal{A}$ and $\mathcal{B}$ are matrices, i.e. second order tensors, (\ref{def:scal_prod}) 
reduces to the standard definition of matrix scalar product.
We next define a tensor-matrix multiplication (i.e. \emph{n-mode product}) and a 
tensor-tensor multiplication (i.e. \emph{$m\choose n$ mode product}). 

 \begin{definition} {\rm \cite[p.~460]{kolda}}\label{def:nmodeprod}
Let $\mathcal{A}\in\mathbb{R}^{I_1\times I_2\times\dots\times I_N}$ and  
  $U\in\mathbb{R}^{J\times I_n}$. The n-mode product 
$\mathcal{A}\times_n U\in\mathbb{R}^{I_1\times\dots\times I_{n-1}\times J\times I_{n+1}\times
  \dots\times I_N}$
is given by
  \begin{equation*}
   (\mathcal{A}\times_n U)_{i_1,\dots, i_{n-1},j,i_{n+1},\dots,I_N}=\sum_{i_n=1}^{I_n} a_{i_1,i_2,\dots,i_n,\dots,i_N}u_{j,i_n}
  \end{equation*}
 \end{definition}
The $n$-mode product satisfies the commutative property when the multiplications are performed 
along different modes.
%

\begin{definition}{\rm \cite{Cichocki2014}}\label{def:mchoosen}
The ${m}\choose {n}$-product of a tensor 
$\mathcal{A}\in\mathbb{R}^{I_1\times I_2\times\dots\times I_N}$ with 
a tensor $\mathcal{B}\in\mathbb{R}^{J_1\times J_2\times\dots\times J_M}$, such that $I_n=J_m$, is defined as
$\mathcal{C}=\mathcal{A}\times_n^m \mathcal{B}$,
where $\mathcal{C}\in\mathbb{R}^{I_1\times\dots\times I_{n-1}\times I_{n+1}\times\dots\times I_N\times 
J_1\times\dots\times J_{m-1}\times J_{m+1}\times\dots\times J_M}$. Its entries are given by
\begin{eqnarray*}
\mathcal{C}(i_1,\dots,i_{n-1},&&i_{n+1},\dots,i_N,j_1,\dots j_{m-1},j_{m+1},\dots,j_M)= \\
&&\displaystyle\sum_{i=1}^{I_n} \mathcal{A}(i_1,\dots,i_{n-1},i,i_{n+1},\dots,i_N)\mathcal{B}(j_1,\dots,j_{m-1},i,j_{m+1},\dots,j_M).
\end{eqnarray*}
\end{definition}

Finally we define the \emph{unfolding} or matricization, i.e. the process of 
flattening a tensor into a matrix. There are several ways to define it, and
we consider the one used by Oseledets in \cite{oseledets}. Let
 $\mathcal{A}\in\mathbb{R}^{I_1\times I_2\times\dots\times I_N}$, 
then the $k$th unfolding of $\mathcal{A}$ is 
 the matrix $A_{[k]}\in\mathbb{R}^{(I_1I_2\dots I_k)\times(I_{k+1}I_{k+2}\dots I_N)}$, 
with elements 
{
\begin{equation}
 A_{[k]}(\overline{i_1\dots i_k},\overline{i_{k+1}\dots i_N})=\mathcal{A}(i_1,\dots,i_k,i_{k+1},\dots, i_N),
 \label{eq:tt_unfolding}
\end{equation}
where $\overline{i_1\dots i_N}=i_1+(i_2-1)I_1+\ldots+(i_N-1)I_1\cdots I_{N-1}$.}

We can thus define the 
{\it Tensor-Train} decomposition of 
an $N$-way array into a product of third order tensors, called TT-cores.
\begin{definition}
\label{def:tt}
 Given an $N$th order tensor tensor $\mathcal{A}$, its Tensor-Train Decomposition is 
given by the following
 \begin{equation}
\begin{split}
 \mathcal{A}(i_1,\dots,i_N)&=G_1(i_1,:)\mathcal{G}_2(:,i_2,:)\mathcal{G}_3(:,i_3,:)\dots G_N(i_N,:)\\
 &= G_1(i_1)\mathcal{G}_2(i_2)\mathcal{G}_3(i_3)\dots G_N(i_N)
 \label{eq_tt}
 \end{split}
\end{equation}
\end{definition}
Using Definition \ref{def:mchoosen} for the ${m}\choose{n}$-mode product, (\ref{eq_tt}) can be written as 
\begin{equation} \label{eq:tt_nmode}
 \mathcal{A}=G_1\times_2^1\mathcal{G}_2\times_3^1\mathcal{G}_3\times_3^1\dots\times_3^1 G_N.
\end{equation}
If $\mathcal{A}$ is a third-order tensor then the Tensor-Train Decomposition reduces 
to the Tucker decomposition,
$\mathcal{A}=\mathcal{G}_2\times_1 G_1\times_3 G_3^T$.
Finally, for general $N$th order tensor, 
the Tensor-Train algorithm requires that the 
matrices $\left(\mathcal{G}_i\right)_{[2]}$,
$i=1,\dots,N-1$  have orthonormal columns.

\end{document}